\documentclass{amsart}

\usepackage{amsmath}
\usepackage{amsfonts}
\usepackage{amsthm} 
\usepackage{amssymb}

\usepackage{graphicx} 

\usepackage{color} 

\usepackage{bm} 

\newtheorem{thm}{Theorem}[section]
\newtheorem{prop}[thm]{Proposition}
\newtheorem{cor}[thm]{Corollary}
\newtheorem{lem}[thm]{Lemma}
\newtheorem{rem}[thm]{Remark}

\newtheorem{defn}[thm]{Definition}
\newtheorem{assump}[thm]{Assumption}

\numberwithin{equation}{section} 

\makeatletter 
\@namedef{subjclassname@2010}{2010 Mathematics Subject Classification} 
\makeatother 

\begin{document}

\title[Oscillatory integrals with phase functions of real powers]
{Oscillatory integrals with phase functions of positive real powers and asymptotic expansions} 

\author{Toshio NAGANO and Naoya MIYAZAKI} 

\address{Department of Liberal Arts, Faculty of Science and Technology, Tokyo University of Science, 2641, Yamazaki, Noda, Chiba 278-8510, JAPAN} 
\email{tonagan@rs.tus.ac.jp}

\address{Department of Mathematics, Faculty of Economics, Keio University, Yokohama, 223-8521, JAPAN} 
\email{miyazaki@a6.keio.jp}

\thanks{The first author was supported by Tokyo University of Science Graduate School doctoral program scholarship 
and an exemption of the cost of equipment from 2016 to 2018 
and would like to thank to Emeritus Professor Minoru Ito for 
giving me an opportunity of studies and preparing the environment.} 


\keywords{Oscillatory integral, the method of stationary phase, asymptotic expansion, Fresnel integral.} 

\subjclass[2010]{Primary 42B20 ; Secondary 41A60, 33B20} 

\date {March 23, 2022} 

\begin{abstract} 
As to methods for expanding an oscillatory integral into an asymptotic series with respect to the parameter, the method of stationary phase for the non-degenerate phases and the method of using resolution of singularities for degenerate phases are known. 
The aim of this paper is to extend the former for degenerate phases with positive real powers without using resolution of singularities. 
For this aim, we first generalize Fresnel integrals as oscillatory integrals with phase functions of positive real powers. 
Next, by using this result, 
we have asymptotic expansions of oscillatory integrals for degenerate phases with positive real powers including moderate oscillations and for a wider amplitude class in one variable. 
Moreover, we obtain asymptotic expansions of oscillatory integrals for degenerate phases consisting of sums of monomials in each variable including the types $A_{k}$, $E_6$, $E_8$ in multivariable. 
\end{abstract} 

\maketitle 

\section{Introduction} 
\label{Introduction} 
The oscillatory integral is one of the improper integrals with conditional convergence, 
which appears as an integral representation of a pseudo-differential operator or a Fourier integral operator 
in constructing a fundamental solution of an initial value problem for a hyperbolic partial differential equation or a time-dependent Schr\"{o}dinger equation. (\cite{Hormander01}, \cite{Kumano-go}, \cite{Fujiwara3}, etc.) 

In general, 
for a real-valued function $\phi$ defined on $\mathbb{R}^{n}$, a complex-valued function $a$ defined on $\mathbb{R}^{n}$, and a positive real number $\lambda$, 
if there exists the following limit of the improper integral independent of $\chi$ in the Schwartz space $\mathcal{S}(\mathbb{R}^{n})$ with $\chi(0) = 1$, 
then it is denoted as 
\begin{align} 
I_{\phi}[a](\lambda) 
:= Os\mbox{-}\int_{\mathbb{R}^{n}} e^{i\lambda \phi(x)} a(x) dx 
:= \lim_{\varepsilon \to +0} \int_{\mathbb{R}^{n}} e^{i\lambda \phi(x)} a(x) \chi (\varepsilon x) dx 
\label{oscillatory_integral_def} 
\end{align} 
and called an oscillatory integral with a phase function $\phi$, an amplitude function $a$, and a parameter $\lambda$. 
In particular, 
if $a \in L^{1}(\mathbb{R}^{n})$, then \eqref{oscillatory_integral_def} coincides with the usual Lebesgue integral $\int_{\mathbb{R}^{n}} e^{i\lambda \phi(x)} a(x) dx$ by Lebesgue's dominated convergence theorem. 

It usually is difficult to compute values of oscillatory integrals except for some examples like Fresnel integrals:
\begin{align} 
\int_{0}^{\infty} e^{\pm ix^{2}} dx 
= \frac{\sqrt{\pi}}{2} e^{\pm i \frac{\pi}{4}}. 
\label{Fresnel_integrals} 
\end{align} 
The double sign $\pm$'s are in the same order from now on, unless otherwise mentioned. 
We then expand oscillatory integrals into asymptotic series with respect to their parameters. 
In \eqref{oscillatory_integral_def}, where the amplitude $a$ hardly changes compared to the phase $\phi$, the value of the integral cancels out, 
while where $\phi$ hardly changes compared to $a$, the one remains.  
Therefore if $\phi$ has a critical point, 
then the main contribution to the integral comes from the neighborhood of the point. 
In particular, 
if the critical point of $\phi$ is non-degenerate (then $\phi$ is called a non-degenerate phase), that is, the quadratic form of the Hessian matrix of $\phi$ at the critical point is non-degenerate, 
then the method for obtaining the asymptotic expansion of an oscillatory integral based on Fresnel integrals and the Morse lemma is called the method of stationary phase (Theorem \ref{stationary_phase_method}). 
The case when the critical point of $\phi$ is degenerate (then $\phi$ is called a degenerate phase) is often discussed in relation to the theory of singularities.  
Malgrange \cite{Malgrange} has indicated the form of asymptotic expansions of oscillatory integrals with real analytic phase functions 
based on the asymptotic expansion of Dirac's distribution with supports on level sets in real analytic hyper surfaces that Jeanquartier \cite{Jeanquartier} has obtained by Hironaka's resolution of singularities theorem \cite{Hi}. 
Var\v{c}henko \cite{Var} has constructed resolution of singularities of some real analytic phase functions by using the toric resolution, 
and has indicated power exponents of asymptotic expansions of oscillatory integrals (Theorem \ref{Varchenko_theorm}). 
Moreover as to recent developments, there are \cite{Greenblatt01}, \cite{Greenblatt02}, \cite{Kamimoto-Nose02}, \cite{Kamimoto-Nose03}, \cite{Kamimoto-Nose04}, etc. 
Thus the form of the asymptotic expansion has been already obtained under several hypotheses. 
However, the coefficients of terms in the asymptotic expansion are not obtained concretely except for the ones of some leading terms 
because it is hard to compute resolution of singularities for degenerate phases in general. 

The aim of this paper is to extend the method of stationary phase for degenerate phases with positive real powers without using resolution of singularities. 
Usually the power exponent in the phase is assumed to be a positive integer because it is the order of zero points of a general phase function. 
In this paper, we mainly consider the following oscillatory integrals in one variable 
\begin{align} 
\tilde{I}^{\pm}_{p,q}[a](\lambda) 
:= Os\text{-}\int_{0}^{\infty} e^{\pm i \lambda x^{p}} x^{q-1} a (x) dx 
:= \lim_{\varepsilon \to +0} \int_{0}^{\infty} e^{\pm i \lambda x^{p}} x^{q-1} a (x) \chi (\varepsilon x) dx 
\notag 
\end{align} 
for any positive real number $p$ and $q$ (Definition \ref{oscillatory_integral_definition}). 
The integral of this type is considered in \cite{Kamimoto-Nose04} for the case when $p$ is a positive rational number (See section \ref{Previous_studies}). 
However, the case when $p$ is a positive real number has not been considered. 
We are particularly interested in the case of $0<p<1$ because then the oscillatory integrals exist even though the phases have no critical point and give ``moderate oscillations.'' 
When we show the existence of an oscillatory integral, we usually use integration by parts after dividing the integral by a cut-off function around the critical point of the phase. 
If $p$ is a positive real number, since the integrals are performed on the half-line $(0,\infty)$, when using the integration by parts, it becomes necessary to note the convergence of the improper integral and the limit of the boundary value at the left endpoint in the half-line $(0,\infty)$. 
We also consider the oscillatory integrals for conditions $\tau \in \mathbb{R}$ and $-1 \leq \delta < p-1$ in the amplitude class $\mathcal{A}^{\tau}_{\delta}(\mathbb{R})$ (Definition \ref{A_tau_delta}) which are wider than the typical ones $\tau \in \mathbb{N}$ and $\delta = -1$. 
These points are different from the standard argument. 

For the aim above, we first extend an exponent $p$ in the phase function $\phi(x) = x^{p}$ from $2$ to any positive real number in one variable. 
The key idea is a generalization of Fresnel integrals \eqref{Fresnel_integrals}. 
As to proofs of Fresnel integrals, 
several ways are known, for example \cite{Sugiura} I p.326, II p.85, 245, etc. 
In the proofs, we especially focus on the way of applying Cauchy's integral theorem 
to a holomorphic function $e^{-iz^{2}}$ on the domain with a fan of the center at the origin of Gaussian plane as a boundary (\cite{Ito-Komatsu} p.23). 
By changing the fan used in the proof with a holomorphic function $e^{-iz^{p}}z^{q-1}$ as an integrand, 
we can generalize Fresnel integrals for $p>q>0$ in the following way: 
\begin{align} 
I_{p,q}^{\pm} 
:= \int_{0}^{\infty} e^{\pm ix^{p}} x^{q-1} dx 
= p^{-1} e^{\pm i\frac{\pi}{2} \frac{q}{p}} \varGamma \left( \frac{q}{p} \right) 
\notag 
\end{align} 
(Lemma \ref{Generalized Fresnel integrals}), where $\varGamma$ is the Gamma function. 
As to the case of $p > 0$ and $q > 0$, 
by making senses of these integrals via oscillatory integrals, 
we obtain 
\begin{align} 
\tilde{I}_{p,q}^{\pm} 
:= Os\mbox{-}\int_{0}^{\infty} e^{\pm ix^{p}} x^{q-1} dx 
= p^{-1} e^{\pm i\frac{\pi}{2} \frac{q}{p}} \varGamma \left( \frac{q}{p} \right). 
\notag 
\end{align} 
Moreover these are extended to meromorphic functions on $\mathbb{C}$ by analytic continuation. 
Then we call $\tilde{I}^{\pm}_{p,q}$ ``generalized Fresnel integrals'' (Theorem \ref{th01}). 

Generalized Fresnel integrals can also be obtained by substituting $z = q/p$ and $\xi = 1$ after using the change of variable $t=x^{p}$ on the well-known formulas: 
\begin{align} 
\lim_{\varepsilon \to +0} \int_{0}^{\infty} e^{\pm i\xi t} t^{z-1} e^{-\varepsilon t} dt 
= e^{\pm i \frac{\pi}{2} z} \varGamma (z) \xi^{-z} 
\label{Hormander02_p_167_Example_7_1_17} 
\end{align} 
for $\xi > 0$ (\cite{Hormander02} p.167 Example 7.1.17) which are used in obtaining the Fourier transform and the inverse Fourier transform of a homogeneous distribution $t_{+}^{z-1}$ of degree $z-1$ where $\Re z > 0$. 
However, the formulas that corresponds to generalized Fresnel integrals are not even written there. 

We write the formulas of generalized Fresnel integrals explicitly and clarify when they hold in the senses of improper integrals or oscillatory integrals as stated above. 

By using generalized Fresnel integrals, 
we can obtain asymptotic expansions of oscillatory integrals with concrete formulas for each term. 

First, suppose $a \in \mathcal{A}^{\tau}_{\delta}(\mathbb{R})$. If $p>0$, then for any $N \in \mathbb{N}$, we have 
\begin{align} 
Os\text{-}\int_{\bold{0}}^{\infty} e^{\pm i\lambda x^{p}} a(x) dx 
&= \sum_{k=0}^{N-1} \tilde{I}_{p,k+1}^{\pm} \frac{a^{(k)}(0)}{k!} \lambda ^{-\frac{k+1}{p}} + O\Big( \lambda^{-\frac{N+1-(p-[p])}{p}} \Big) 
\label{unit} 
\end{align} 
as $\lambda \to \infty$. It is new to show that this expansion holds even when $p$ is any positive real number close to $0$. 
If $m \in \mathbb{N}$, then for any $N \in \mathbb{N}$, we have 
\begin{align} 
Os\text{-}\int_{\bm{-\infty}}^{\infty} e^{\pm i\lambda x^{m}} a(x) dx 
= \sum_{k=0}^{N-1} c_{m,k}^{\pm} \frac{a^{(k)}(0)}{k!} \lambda ^{-\frac{k+1}{m}} + O\Big( \lambda^{-\frac{N+1}{m}} \Big) 
\label{the_unified_formula} 
\end{align} 
as $\lambda \to \infty$, 
with 
\begin{align} 
c_{m,k}^{\pm} 
:&= \tilde{I}_{m,k+1}^{\pm} + (-1)^{k} \tilde{I}_{m,k+1}^{\pm \pm^{m}}, 
\label{c_mk^pm03} 
\end{align} 
where 
\begin{align} 
\tilde{I}_{m,k+1}^{\pm \pm^{m}} 
:= Os\mbox{-}\int_{0}^{\infty} e^{\pm (-1)^{m} iy^{m}} y^{k} dy 
= m^{-1} e^{\pm (-1)^{m} i\frac{\pi}{2} \frac{k+1}{m}} \varGamma \left( \frac{k+1}{m} \right) 
\notag 
\end{align} 
(Theorem \ref{th02}). 
Moreover if $m=2l-1$ and $m=2l$ for $l \in \mathbb{N}$, then for any $N \in \mathbb{N}$, we have 
\begin{align} 
&Os\text{-}\int_{-\infty}^{\infty} e^{\pm i\lambda x^{2l-1}} a(x) dx 
= \frac{2}{2l-1} \sum_{k=0}^{N-1} \left\{ \cos \frac{\pi (2k+1)}{2(2l-1)} \varGamma \left( \frac{2k+1}{2l-1} \right) \frac{a^{(2k)}(0)}{(2k)!} \lambda ^{-\frac{2k+1}{2l-1}} \right. \notag \\ 
&\hspace{2.75cm}\left. \pm i \sin \frac{\pi (2k+2)}{2(2l-1)} \varGamma \left( \frac{2k+2}{2l-1} \right) \frac{a^{(2k+1)}(0)}{(2k+1)!} \lambda ^{-\frac{2k+2}{2l-1}} \right\} + O\Big( \lambda ^{-\frac{N+1}{2l-1}} \Big), \label{odd_formula} \\ 
&Os\text{-}\int_{-\infty}^{\infty} e^{\pm i\lambda x^{2l}} a(x) dx 
= \frac{1}{l} \sum_{k=0}^{N-1} e^{\pm i\frac{\pi}{2} \frac{2k+1}{2l}} \varGamma \left( \frac{2k+1}{2l} \right) \frac{a^{(2k)}(0)}{(2k)!} \lambda ^{-\frac{2k+1}{2l}} + O\Big( \lambda ^{-\frac{N}{l}} \Big) \label{even_formula} 
\end{align} 
as $\lambda \to \infty$ (Corollary \ref{cor01}). 
These formulas are written in \cite{Hormander02} as the Equations (7.7.30) and (7.7.31) if $a(x) \in C^{\infty}_{0}(\mathbb{R})$: \begin{align} 
&\int_{-\infty}^{\infty} e^{i\lambda x^{2l-1}} a(x) dx 
= \frac{2}{2l-1} \sum_{k=0}^{N-1} \left\{ \sin \frac{(2l-2)(k+1) \pi}{2l-1} \right\} \varGamma \left( \frac{k+1}{2l-1} \right) \frac{(-1)^{k} a^{(k)}(0)}{k!} \notag \\ 
&\hspace{3.5cm}\times \lambda^{-\frac{k+1}{2l-1}} + O\Big( \lambda^{-\frac{N+1}{2l-1}} \Big), 
\label{Hormander02_Equations_7_7_30} \\ 
&\int_{-\infty}^{\infty} e^{i\lambda x^{2l}} a(x) dx 
= \frac{1}{l} \sum_{k=0}^{N-1} e^{i \frac{\pi}{2} \frac{2k+1}{2l}} \varGamma \left( \frac{2k+1}{2l} \right) \frac{a^{(2k)}(0)}{(2k)!} \lambda^{-\frac{2k+1}{2l}} 
+ O\Big( \lambda^{-\frac{2N+1}{2l}} \Big) 
\label{Hormander02_Equations_7_7_31} 
\end{align} 
as $\lambda \to \infty$ in our notation. 
However, the unified formula of \eqref{Hormander02_Equations_7_7_30} and \eqref{Hormander02_Equations_7_7_31} corresponding to \eqref{the_unified_formula} is not written and the relation between \eqref{Hormander02_p_167_Example_7_1_17} and \eqref{Hormander02_Equations_7_7_30},  \eqref{Hormander02_Equations_7_7_31} is not mentioned either there. 
We indicate the unified formula \eqref{the_unified_formula} deduces \eqref{odd_formula} and \eqref{even_formula}. 
Thus considering a positive exponent $p$ in real number makes it possible to reconsider the well-known formulas from a unified perspective, which makes it easier to extend the unified formula to multivariable cases. 
We also note that our remainder term is more accurate than one of the formula by $1/m$ order for even $m$ as above. 

Moreover considering a positive exponent $p$ in real number allows a wider change of variable. 
As the applications, the phase function $x^{p}$ in \eqref{unit} can be extended to $x^{p} (1 + \sum_{j=1}^{\infty} a_{j} x^{j})$ such that $l_{0} := \sup_{j \in \mathbb{N}} |a_{j}|^{1/j} < \infty$. 
In fact, let $U$ be a connected open neighborhood of the origin in $\mathbb{R}$ such that $U \subset (-R_{0},R_{0})$ where $R_{0} := (2l_{0})^{-1}$ and $R_{0}=\infty$ if $l_{0}=0$. 
Then there exists a diffeomorphism $x = \varPhi (y)$ of class $C^{\infty}$ for $x,y \in U$ such that $x (1 + \sum_{j=1}^{\infty} a_{j} x^{j})^{1/p} = y$, $\varPhi (U \cap [0,\infty)) = U \cap [0,\infty)$ and $\varPhi (U \cap (-\infty,0]) = U \cap (-\infty,0]$, and for any $a \in C^{\infty}_{0}(\mathbb{R})$ such that $\mathrm{supp}~a \subset U$ and for any $N \in \mathbb{N}$, 
\begin{align} 
&\int_{0}^{\infty} e^{\pm i\lambda x^{p} \left( 1 + \sum_{j=1}^{\infty} a_{j} x^{j} \right)} a(x) dx \notag \\ 
&= \sum_{k=0}^{N-1} 
\frac{\tilde{I}_{p,k+1}^{\pm}}{k!} \left( \frac{d}{dy} \right)^{k} \bigg|_{y=0} \left\{ a(\varPhi (y)) \frac{d\varPhi}{dy}(y) \right\} \lambda ^{-\frac{k+1}{p}} 
+ O\left( \lambda ^{-\frac{N+1-(p-[p])}{p}} \right) 
\notag 
\end{align} 
as $\lambda \to \infty$ (Theorem \ref{analytic_phase} (i)). 
If $m \in \mathbb{N}$, then for any $N \in \mathbb{N}$, 
\begin{align} 
&\int_{-\infty}^{\infty} e^{\pm i\lambda x^{m} \left( 1 + \sum_{j=1}^{\infty} a_{j} x^{j} \right)} a(x) dx \notag \\ 
&= \sum_{k=0}^{N-1} 
\frac{c_{m,k}^{\pm -}}{k!} \left( \frac{d}{dy} \right)^{k} \bigg|_{y=0} \left\{ a(\varPhi (y)) \frac{d\varPhi}{dy}(y) \right\} \lambda ^{-\frac{k+1}{m}} + O\left( \lambda ^{-\frac{N+1}{m}} \right) 
\notag 
\end{align} 
as $\lambda \to \infty$, where 
\begin{align} 
c_{m,k}^{\pm -} 
:&= \tilde{I}_{m,k+1}^{\pm} - (-1)^{k} \tilde{I}_{m,k+1}^{\pm \pm^{m}} 
\notag 
\end{align} 
(Theorem \ref{analytic_phase} (ii)). 
Here, when $a_{j}=1/j!$, by using the Lambert W function $X=W_{0}(Y)$, 
we can write $\varPhi (y) = pW_{0} \big( \frac{y}{p} \big)$ and $\frac{d\varPhi}{dy}(y) = \frac{dW_{0}}{dY} \big( \frac{y}{p} \big)$ 
(Corollary \ref{Lambert_W_function}). 

As to applications in the case of multivariable, 
we obtain the asymptotic expansions of oscillatory integrals for $a \in C^{\infty}_{0}(\mathbb{R}^{n})$. 
If $p:=(p_{1},\dots,p_{n}) \in (0,\infty)^{n}$, then for any $N \in \mathbb{N}$, we have 
\begin{align} 
&\int_{(0,\infty)^{n}} e^{i\lambda \sum_{j=1}^{n} \pm_{j} x_{j}^{p_{j}}} a(x) dx \notag \\ 
&= \sum_{\alpha \in \Omega_{p}^{N}} \prod_{j=1}^{n} \tilde{I}^{\pm_{j}}_{p_{j},\alpha_{j}+1} \frac{\partial_{x}^{\alpha} a(0)}{\alpha!} \lambda^{-\sum_{j=1}^{n} \frac{\alpha_{j}+1}{p_{j}}} + O \Big( \lambda^{-\frac{N+1-\max (p_{j}-[p_{j}])}{\max p_{j}}} \Big) 
\notag 
\end{align} 
as $\lambda \to \infty$, 
where $\pm_{j}$ represents ``$+$" or ``$-$" for each $j$, and 
\begin{align} 
\Omega_{p}^{N} 
:= \bigg\{ \alpha = (\alpha_{1},\dots,\alpha_{n}) \in \mathbb{Z}_{\geq 0}^{n} \bigg| \sum_{j=1}^{n} \frac{\alpha_{j}+1}{p_{j}} < \frac{N+1-\max (p_{j}-[p_{j}])}{\max p_{j}} \bigg\}. 
\notag 
\end{align} 
If $m:=(m_{1},\dots,m_{n}) \in \mathbb{N}^{n}$, then for any $N \in \mathbb{N}$, we have 
\begin{align} 
\int_{\mathbb{R}^{n}} e^{i\lambda \sum_{j=1}^{n} \pm_{j} x_{j}^{m_{j}}} a(x) dx 
&= \sum_{\alpha \in \Omega_{m}^{N}} \prod_{j=1}^{n} c^{\pm_{j}}_{m_{j},\alpha_{j}} \frac{\partial_{x}^{\alpha} a(0)}{\alpha!}  \lambda^{-\sum_{j=1}^{n} \frac{\alpha_{j}+1}{m_{j}}} + O \Big( \lambda^{-\frac{N+1}{\max m_{j}}} \Big) 
\notag 
\end{align} 
as $\lambda \to \infty$, 
with 
\begin{align} 
c^{\pm_{j}}_{m_{j},\alpha_{j}} 
:&= \tilde{I}_{m_{j},\alpha_{j}+1}^{\pm_{j}} + (-1)^{\alpha_{j}} \tilde{I}_{m_{j},\alpha_{j}+1}^{\pm_{j} \pm^{m_{j}}} 
\notag 
\end{align} 
(Theorem \ref{multivariable}). 
These results include the cases of $A_{k}$, $E_6$, $E_8$-phase functions: 
\begin{align} 
&A_{k} : \pm_{1} x_{1}^{k+1} + \sum_{j=2}^{n} \pm_{j} x_{j}^{2}, & 
&E_{6} : x_{1}^{3} \pm_{2} x_{2}^{4} + \sum_{j=3}^{n} \pm_{j} x_{j}^{2}, & 
&E_{8} : x_{1}^{3} + x_{2}^{5} + \sum_{j=3}^{n} \pm_{j} x_{j}^{2}, 
\notag 
\end{align} 
where $k \in \mathbb{N}$ (\cite{Arnold01}, \cite{AGV}). 
These cases are considered in \cite{Duistermaat02}. 
However, the concrete expression of the asymptotic expansion in each case is not showed there. 
We have also clarified how the principal part is determined for the order of the remainder term in multivariate asymptotic expansions by defining its index set $\Omega_{p}^{N}$. This is a new result. 

Thus considering oscillatory integrals with phase functions of positive real powers enables us accurately to represent the coefficients of each term in the asymptotic expansions of oscillatory integrals decomposed on the half-line by using generalized Fresnel integrals. 
Hence we can expect to obtain more concrete and detailed results. 
We think it is desirable to extend the range of applications for such computable cases while comparing with the results using resolution of singularities.

For the aim above, 
after referring to earlier studies on asymptotic expansions of oscillatory integrals in \S \ref{Previous_studies}, 
we first define the oscillatory integrals discussed in this paper in \S \ref{Preliminary}. 

In \S \ref{section_Existence of oscillatory integrals}, we show the existence of the oscillatory integrals. 

In \S \ref{section_Generalized_Fresnel_Integrals}, we extend Fresnel integrals 
by changing of a path for integration 
in the well-known proof using Cauchy's integral theorem. 
Then, according to oscillatory integrals, 
we also obtain further generalization of Fresnel integrals. 

Furthermore, in \S \ref{Applications_to_asymptotic expansions}, 
according to generalized Fresnel integrals, we establish 
the asymptotic expansions of oscillatory integrals with phase functions of positive real powers 
in one variable and multivariable. 

To the end of \S \ref{Introduction}, we note notation used in this paper: 

$\alpha = (\alpha_{1},\dots,\alpha_{n}) \in \mathbb{Z}_{\geq 0}^{n}$ 
is a multi-index with a length 
$| \alpha | = \alpha_{1} + \cdots + \alpha_{n}$, 
and then, we use 
$x^{\alpha} = x_{1}^{\alpha_{1}} \cdots x_{n}^{\alpha_{n}}$, 
$\alpha! = \alpha_{1}! \cdots \alpha_{n}!$, 
$\partial_{x}^{\alpha} = \partial_{x_{1}}^{\alpha_{1}} \cdots \partial_{x_{n}}^{\alpha_{n}}$ 
and $D_{x}^{\alpha} = D_{x_{1}}^{\alpha_{1}} \cdots D_{x_{n}}^{\alpha_{n}}$, 
where $\partial_{x_{j}} = \frac{\partial}{\partial x_{j}}$ 
and $D_{x_{j}} = i^{-1} \partial_{x_{j}}$ for $x = (x_{1}, \dots, x_{n})$. 

$C^{\infty}(\mathbb{R}^{n})$ is the set of complex-valued functions of class $C^{\infty}$ defined on $\mathbb{R}^{n}$. 
$C^{\infty}_{0}(\mathbb{R}^{n})$ is the set of $f \in C^{\infty}(\mathbb{R}^{n})$ with compact support. 
$\mathcal{S}(\mathbb{R}^{n})$ is the Schwartz space of rapidly decreasing functions of class $C^{\infty}$ defined on $\mathbb{R}^{n}$, 
that is, the Fr\'{e}chet space of $f \in C^{\infty}(\mathbb{R}^{n})$ 
such that $\sup_{x \in \mathbb{R}^{n}} \langle x \rangle^{k} | \partial_{x}^{\alpha} f (x) | < \infty$ for any $k \in \mathbb{Z}_{\geq 0}$ and for any multi-index $\alpha \in \mathbb{Z}_{\geq 0}^{n}$ where $\langle x \rangle := (1+|x|^{2})^{1/2}$. 

$\varGamma$ is the Gamma function. 

$[x]$ is the Gauss' symbol for $x \in \mathbb{R}$, that is, $[x] \in \mathbb{Z}$ such that $x-1 < [x] \leq x$. 

$[x)$ is the greatest integer smaller than real number $x$, that is, $x-1 \leq [x) < x$. 

$O$ means the Landau's symbol, that is, 
$f(x) = O(g(x))~(x \to a)$ if $|f(x)/g(x)|$ is bounded on $\{ x \in \mathbb{R} | 0<|x-a|<\varepsilon \}$ for some $\varepsilon > 0$ if $a \in \mathbb{R}$, on $(c,+\infty)$ for some $c \in \mathbb{R}$ if $a = \infty$, or on $(-\infty,c)$ for some $c \in \mathbb{R}$ if $a = -\infty$, for complex-valued functions $f$ and $g$ defined on $D \subset \mathbb{R} \cup \{ \pm \infty \}$ and $a \in \bar{D}$. 

$\delta_{ij}$ is the Kronecker's delta, that is, $\delta_{ii} = 1$, and $\delta_{ij} = 0$ if $i \ne j$. 

$t^{+} := \max \{ t,0 \}$ for $t \in \mathbb{R}$. 

The double sign $\pm$'s are in the same order. 

$\pm^{m}$ stands for a sign of $(-1)^{m}$ for $m \in \mathbb{N}$, that is, $(-1)^{m} = \pm^{m} 1$. 

$\pm_{j}$ represents ``$+$" or ``$-$" for each $j \in \mathbb{N}$. 

\section{Earlier studies} 
\label{Previous_studies} 

In this section, we shall refer to earlier studies on asymptotic expansions of oscillatory integrals. 

When a phase function has a non-degenerate critical point, then the method of stationary phase holds as follows (\cite{Hormander01}, \cite{Hormander02}, \cite{Fujiwara3}, etc.): 
\begin{thm} 
\label{stationary_phase_method} 
Let $\phi$ be a real-valued function of class $C^{\infty}$ defined on $\mathbb{R}^{n}$ 
with a non-degenerate critical point $\bar{x}$ 
and $a \in C^{\infty}_{0}(\mathbb{R}^{n})$. 
Then 
there exist neighborhoods $V$ of $\bar{x}$ 
and $W$ of the origin in $\mathbb{R}^{n}$, 
and diffeomorphism 
$x=\varPhi (y)$ of class $C^{\infty}$ 
for $x \in V$ and $y \in W$, 
and for any $N \in \mathbb{N}$, 
\begin{align} 
&\int_{\mathbb{R}^{n}} e^{i\lambda \phi (x)} a(x) dx 
= (2\pi)^{\frac{n}{2}} \frac{e^{i \frac{\pi}{4} \mathrm{sgn} \mathrm{Hess} \phi (\bar{x})}}{|\det \mathrm{Hess} \phi (\bar{x})|^{\frac{1}{2}}} e^{i\lambda \phi(\bar{x})} \notag \\ 
&\times \sum_{k=0}^{N-1} \frac{1}{k!} \Big( -i\frac{1}{2} \langle \mathrm{Hess} \phi (\bar{x})^{-1} D_{y},D_{y} \rangle \Big)^{k} 
\bigg|_{y=0} a(\varPhi(y)) J_{\varPhi}(y) \lambda^{-k-\frac{n}{2}} + O\left( \lambda^{-N-\frac{n}{2}} \right) 
\notag 
\end{align} 
as $\lambda \to \infty$, 
where $\mathrm{Hess} \phi (\bar{x}) := (\partial ^{2} \phi (\bar{x})/\partial x_{i} \partial x_{j})_{i,j=1,\dots,n}$ is the Hessian matrix of $\phi$ at $\bar{x}$ with ``$p$" positive and ``$n-p$" negative eigenvalues, 
$\mathrm{sgn} \mathrm{Hess} \phi (\bar{x}) := p-(n-p)$, 
and $J_{\varPhi}(y)$ is the Jacobian of $\varPhi$ at $y \in W$. 
\end{thm} 

In order to describe the case when a phase function has a degenerate critical point, 
according to typical notation, 
we consider an oscillatory integral 
\begin{align} 
I_{f}[\varphi](\tau) 
:= \int_{\mathbb{R}^{n}} e^{i \tau f(x)} \varphi(x) dx 
\notag 
\end{align} 
under the following assumptions:  
\begin{assump} 
\label{assumption01} 
Assume that $f$ and $\varphi$ satisfy the the following assumptions: 
\begin{align} 
\rm{(1)}~&\text{$f$ is a real-valued function of class $C^{\infty}$ defined on an open neighborhood $U$} \notag \\  
&\text{of the origin in $\mathbb{R}^{n}$ satisfying $f(0) = 0$ and $\nabla f(0) = 0$.} \notag \\ 
\rm{(2)}~&\text{$\varphi \in C^{\infty}_{0}(\mathbb{R}^{n})$ satisfies $\mathrm{supp} \varphi \subset U$.} \notag 
\end{align} 
\end{assump} 

First, we shall recall Newton polyhedra and the related notions. 
\begin{defn} 
Suppose that $f = \sum_{\alpha \in \mathbb{Z}_{\geq 0}^{n}} c_{\alpha} x^{\alpha}$ is a formal power series. 
Then 
Newton polyhedron and Newton diagram of $f$ are defined by 
\begin{align} 
&\Gamma_{+}(f) := \bigg( \text{the convex hull of $\bigcup_{c_{\alpha} \ne 0} (\alpha + \mathbb{R}_{\geq 0}^{n})$} \bigg), \notag \\ 
&\Gamma(f) := (\text{the union of compact faces of $\Gamma_{+}(f)$}) \big) 
\notag 
\end{align} 
respectively, where a set $\gamma$ is called face of $\Gamma_{+}(f)$ if there exists $(a,l) \in \mathbb{Z}_{\geq 0}^{n} \times \mathbb{Z}_{\geq 0}$such that $\gamma = \Gamma_{+}(f) \cap \partial H_{+}(a,l)$ and $\Gamma_{+}(f) \subset H_{+}(a,l) := \{ x \in \mathbb{R}^{n} | \langle a,x \rangle \geq l \}$. 

The principal part of $f$ is defined by $f_{\Gamma(f)} := \sum_{\alpha \in \Gamma(f)} c_{\alpha} x^{\alpha}$, 
The $\gamma$-part of $f$ is defined by $f_{\gamma} := \sum_{\alpha \in \gamma} c_{\alpha} x^{\alpha}$ for face $\gamma$ of $\Gamma_{+}(f)$, 
and $f_{\Gamma(f)}$ is said to be non-singular for $\Gamma(f)$ 
if for any closed face $\gamma$ in $\Gamma(f)$, the $\gamma$-part $f_{\gamma}$ satisfy $\nabla f_{\gamma} \ne 0$ for $x_{1} \cdots x_{n} \ne 0$. 
\end{defn} 

When $f$ is analytic and has a Taylor series with the non-singular principal part at the origin under the Assumption \ref{assumption01}, 
Var\v{c}henko \cite{Var} has constructed resolution of singularity of $f$ by using the toric resolution, 
and has obtained the asymptotic expansions of oscillatory integrals: 
\begin{thm}
\label{Varchenko_theorm} 
Suppose that $f$ is a real analytic function expressed by Taylor series $f = \sum_{\alpha \in \mathbb{Z}_{\geq 0}^{n}} c_{\alpha} x^{\alpha}$ at the origin under the Assumption \ref{assumption01}. 
If the principal part $f_{\Gamma(f)} := \sum_{\alpha \in \Gamma(f)} c_{\alpha} x^{\alpha}$ of $f$ is non-singular for $\Gamma(f)$, 
then the following oscillatory integral has an asymptotic expansion of the form 
\begin{align} 
I_{f}[\varphi](\tau) 
:= \int_{\mathbb{R}^{n}} e^{i \tau f(x)} \varphi(x) dx 
\sim e^{i \tau f(0)} \sum_{p} \sum_{k=0}^{n-1} c_{p,k}(\varphi) \tau^{p} (\log \tau)^{k} 
\label{Varchenko_expansion} 
\end{align} 
$\tau \to \infty$, where $p$ runs through a finite set of arithmetic progressions independent of $\varphi$ whose terms are negative rational numbers determined by the toric resolution of $f$. 
\end{thm} 

When $f$ has Taylor series $\sum_{\alpha \in \mathbb{Z}_{\geq 0}^{n}} c_{\alpha} x^{\alpha}$ at the origin under the Assumption \ref{assumption01}, 
Kamimoto-Nose \cite{Kamimoto-Nose03} have defined ``$f$ admits the $\gamma$-part'' on $U$ if for any $x \in U$ and for any $(a,l) \in \mathbb{Z}_{\geq 0}^{n} \times \mathbb{Z}_{\geq 0}$ with $a = (a_{1},\dots,a_{n})$ such that $\gamma = \Gamma_{+}(f) \cap \partial H_{+}(a,l)$, there exists $\lim_{t \to 0} t^{-l} f(t^{a_{1}} x_{1},\dots,t^{a_{n}} x_{n})$, 
which denoted by $f_{\gamma}$ and called ``the $\gamma$-part'' of $f$ on $U$. 
Then they have introduced the class 
\begin{align} 
\hat{\mathcal{E}}(U) 
:= \{ f \in C^{\infty}(\mathbb{R}^{n})~|~\text{$f$ admits the $\gamma$-part on $U$ for any face $\gamma$ in $\Gamma_{+}(f)$} \} 
\notag 
\end{align} 
including real analytic functions, 
and have indicated \eqref{Varchenko_expansion} holds for any $f \in \hat{\mathcal{E}}(U)$ with the non-degenerate principal part by constructing the toric resolution of singularities. 
Moreover when $f$ has Puiseux series $\sum_{\alpha \in \mathbb{Z}_{\geq 0}^{n}} c_{\alpha} x^{\alpha/p}$ at the origin where $\alpha/p := (\alpha_{1}/p_{1},\dots,\alpha_{n}/p_{n})$ for $p:=(p_{1},\dots,p_{n}) \in \mathbb{N}^{n}$ under the Assumption \ref{assumption01}, 
Kamimoto-Nose \cite{Kamimoto-Nose04} have defined the suitable class $\hat{\mathcal{E}}_{1/p}(U_{+})$ where $U_{+}:=U \cap [0,\infty)^{n}$ and indicate \eqref{Varchenko_expansion} holds for $I_{f}[\varphi](\tau) := \int_{[0,\infty)^{n}} e^{i \tau f(x)} \varphi(x) dx $. 

\section{Preliminary} 
\label{Preliminary} 

In this section, 
we define the oscillatory integrals discussed in this paper. 

First we shall recall a convergence factor, which is a family of functions with properties of Gaussian function $e^{-|x|^{2}}$. 
\begin{prop} 
\label{chi epsilon} 
Let $\chi \in \mathcal{S}(\mathbb{R}^{n})$ with $\chi(0) = 1$. 
Then 
\begin{enumerate} 
\item[(i)] 
$\chi(\varepsilon x) \to 1$ uniformly on any compact set in $\mathbb{R}^{n}$ as $\varepsilon \to +0$. 
\item[(ii)] 
For each multi-index $\alpha \in \mathbb{Z}_{\geq 0}^{n}$, 
there exists a positive constant $C_{\alpha}$ independent of $0 < \varepsilon <1 $ 
such that 
for any $x \in \mathbb{R}^{n}$, 
\begin{align} 
| \partial_{x}^{\alpha} (\chi(\varepsilon x)) | \leq C_{\alpha} \langle x \rangle^{-|\alpha|}. \notag 
\end{align} 
\item[(iii)] 
For any multi-index $\alpha \in \mathbb{Z}_{\geq 0}^{n}$ with $\alpha \ne 0$, 
$\partial_{x}^{\alpha} (\chi(\varepsilon x)) \to 0$ uniformly in $\mathbb{R}^{n}$ as $\varepsilon \to +0$. 
\end{enumerate} 
\end{prop} 

\begin{proof} 
See \cite{Kumano-go} p.47. 
\end{proof} 

Next we define the amplitude function class in one variable as follows: 
\begin{defn} 
\label{A_tau_delta} 
Assume that $p>0$. 
Let $\tau \in \mathbb{R}$ and $-1 \leq \delta < p-1$. 
We say that $a \in C^{\infty}(\mathbb{R})$ belongs to the class $\mathcal{A}^{\tau}_{\delta}(\mathbb{R})$ if 
for each $k \in \mathbb{Z}_{\geq 0}$, there exists a positive constant $C_{k}$ such that for any $x \in \mathbb{R}$, 
\begin{align} 
|a^{(k)}(x)| \leq C_{k} \langle x \rangle^{\tau + \delta k}. 
\notag  
\end{align} 
Then for any $l \in \mathbb{Z}_{\geq 0}$, we set 
\begin{align} 
|a|^{(\tau)}_{l} 
:= \max_{k=0,\dots,l} \sup_{x \in \mathbb{R}} \langle x \rangle^{-\tau - \delta k} |a^{(k)}(x)|. 
\label{|a|^(tau)_l} 
\end{align} 
Then for any $k=0,\dots,l$, 
\begin{align} 
|a^{(k)}(x)| \leq |a|^{(\tau)}_{l} \langle x \rangle^{\tau + \delta k}. 
\label{A_tau_delta_def02} 
\end{align} 
\end{defn} 
\begin{rem} \label{rem01} 
We see the following immediately: 
\begin{align} 
C^{\infty}_{0}(\mathbb{R}) \subset 
\mathcal{S}(\mathbb{R}) 
= \bigcap_{\tau \in \mathbb{R}} \bigcap_{-1 \leq \delta < p-1} \mathcal{A}^{\tau}_{\delta}(\mathbb{R})
\subset \mathcal{A}^{\tau}_{\delta}(\mathbb{R}), 
\label{inclusion_rel_01} 
\end{align} 
because if $a \in \mathcal{S}(\mathbb{R})$, for $\tau \in \mathbb{R}$, $-1 \leq \delta < p-1$ and $k \in \mathbb{Z}_{\geq 0}$, taking $m \in \mathbb{Z}_{\geq 0}$ such that $\tau + \delta k > -m$, since $|a^{(k)}(x)| \leq C_{k} \langle x \rangle^{\tau + \delta k}$ for $x \in \mathbb{R}$ where $C_{k} = \sup_{x \in \mathbb{R}} \langle x \rangle^{m} |a^{(k)}(x)| > 0$, then $a \in \mathcal{A}^{\tau}_{\delta}(\mathbb{R})$, and if $a \in \mathcal{A}^{\tau}_{\delta}(\mathbb{R})$ for any $\tau \in \mathbb{R}$ and $-1 \leq \delta < p-1$, for $m,k \in \mathbb{Z}_{\geq 0}$, taking $\tau \in \mathbb{R}$ and $-1 \leq \delta < p-1$ such that $-\tau - \delta k > m$, since $\sup_{x \in \mathbb{R}} \langle x \rangle^{m} |a^{(k)}(x)| \leq \sup_{x \in \mathbb{R}} \langle x \rangle^{-\tau - \delta k} |a^{(k)}(x)| < \infty$, then $a \in \mathcal{S}(\mathbb{R})$. 
Also if $a \in \mathcal{A}^{\tau}_{\delta}(\mathbb{R})$, then for any $j \in \mathbb{Z}_{\geq 0}$, $a^{(j)} \in \mathcal{A}^{\tau + \delta j}_{\delta}(\mathbb{R})$ and for any $l \in \mathbb{Z}_{\geq 0}$, 
\begin{align} 
|a^{(j)}|^{(\tau + \delta j)}_{l} 
&= \max_{k'=j,\dots,j+l} \sup_{x \in \mathbb{R}} \langle x \rangle^{-\tau - \delta k'} |a^{(k')}(x)| 
\leq |a|^{(\tau)}_{j+l}. 
\label{A_tau_delta_def03} 
\end{align} 
\end{rem} 

Finally we define the oscillatory integrals discussed in this paper as follows: 
\begin{defn} 
\label{oscillatory_integral_definition} 
Assume that $\lambda > 0$, $p>0$, $q>0$, $m \in \mathbb{N}$ and $a \in C^{\infty}(\mathbb{R})$. 
If there exist the following limits of improper integrals independent of $\chi \in \mathcal{S}(\mathbb{R})$ with $\chi(0) = 1$, then we define oscillatory integrals by 
\begin{align} 
\tilde{I}^{\pm}_{p,q}[a](\lambda) 
:&= Os\text{-}\int_{0}^{\infty} e^{\pm i \lambda x^{p}} x^{q-1} a (x) dx 
:= \lim_{\varepsilon \to +0} \int_{0}^{\infty} e^{\pm i \lambda x^{p}} x^{q-1} a (x) \chi (\varepsilon x) dx, 
\label{oscillatory_integral_def01} \\ 
\tilde{J}_{m,1}^{\pm}[a](\lambda) 
:&= Os\text{-}\int_{-\infty}^{\infty} e^{\pm i\lambda x^{m}} a(x) dx 
:= \lim_{\varepsilon \to +0} \int_{-\infty}^{\infty} e^{\pm i \lambda x^{m}} a (x) \chi (\varepsilon x) dx. 
\label{oscillatory_integral_def02} 
\end{align} 
We write $\chi_{\varepsilon}(x) := \chi(\varepsilon x)$ for $x \in \mathbb{R}$ and $0 < \varepsilon <1 $ from now on. 
Then $\chi_{\varepsilon}^{(k)}(x) = \partial_{x}^{k} (\chi(\varepsilon x)) = \partial_{y}^{k} \chi(\varepsilon x)  \varepsilon^{k}$ for $k \in \mathbb{Z}_{\geq 0}$. 
\end{defn} 

\section{Existence of oscillatory integrals} 
\label{section_Existence of oscillatory integrals} 

In this section, 
we shall show the existence of the oscillatory integrals \eqref{oscillatory_integral_def01} and \eqref{oscillatory_integral_def02} for $a \in \mathcal{A}^{\tau}_{\delta}(\mathbb{R})$. 

The following lemma verifies that we can perform integration by parts in the improper integrals on the half-line $(0,\infty)$: 
\begin{lem}[\cite{Nagano-Miyazaki04}] 
\label{L_star_l} 
Assume that 
$\lambda > 0$, $p > 0$, $q > 0$, $a \in \mathcal{A}^{\tau}_{\delta}(\mathbb{R})$ and $\chi_{\varepsilon}(x) := \chi(\varepsilon x)$ for $x \in \mathbb{R}$ where $\chi \in \mathcal{S}(\mathbb{R})$ with $\chi(0) = 1$ and $0 < \varepsilon <1 $. 
Let 
$\varphi \in C^{\infty}_{0}(\mathbb{R})$ be a cut-off function such that $\varphi \equiv 1$ on $|x| \leq r_{0}$ with $r_{0} \geq 1$ and $\varphi \equiv 0$ on $|x| \geq r_{1} > r_{0}$, 
and $a_{h} := a \psi_{h}$ where $\psi_{h} := 1 - \delta_{h1} \varphi$ for $h=0,1$ where $\delta_{h1}$ is the Kronecker's delta, and let $L^{*} := - \frac{1}{i\lambda} \frac{d}{dx}  \frac{1}{px^{p-1}}$ be the formal adjoint operator of $L := \frac{1}{px^{p-1}} \frac{1}{i\lambda} \frac{d}{dx}$ and $l_{0}:=[q/p)$.
Then for any $k \in \mathbb{Z}_{\geq 0}$, 
the following hold: 

\begin{enumerate} 
\item[(i)] 
For any $l \in \mathbb{Z}_{\geq 0}$, let $C_{l,j}$ be a real constant satisfying 
\begin{align} 
&C_{l,j} = (q-pl+j) C_{l-1,j} + C_{l-1,j-1}~\text{for}~j=1,\dots,l-1~\text{with}~l \geq 1, 
\label{C_l_j_def01} \\ 
&C_{l,0} = (q-pl) C_{l-1,0},~C_{l,l} = C_{l-1,l-1}~\text{for}~l \in \mathbb{N}~\text{and}~C_{0,0} = 1. 
\label{C_l_j_def02} 
\end{align} 
Then we have 
\begin{align} 
L^{*l} (x^{q-1} a_{h}(x) \chi_{\varepsilon}^{(k)}(x)) 
= \left( \frac{i}{\lambda p} \right)^{l} \sum_{j=0}^{l} C_{l,j} x^{q-1-pl+j} (a_{h}(x) \chi_{\varepsilon}^{(k)}(x))^{(j)} 
\label{L star_l_01} 
\end{align} 
for $x \in (0,\infty)$ and $h=0,1$, 
where $L^{*0}$ is an identity operator, 
and 
\begin{align} 
\text{$C_{l,0} = \prod_{s=1}^{l} (q-ps)$ for $l \in \mathbb{N}$ and $C_{l,l} = 1$ for $l \in \mathbb{Z}_{\geq 0}$.} 
\label{C_l_0_02} 
\end{align} 

\item[(ii)] 
If $q>p$ and $h=0$, 
then for any $l=0,\dots,l_{0}$, or if $p > 0$, $q > 0$ and $h = 1$, then for any $l \in \mathbb{Z}_{\geq 0}$, 
the improper integrals 
\begin{align} 
\int_{0}^{\infty} e^{\pm i\lambda x^{p}} (\pm L^{*})^{l} (x^{q-1} a_{h}(x) \chi_{\varepsilon}^{(k)}(x)) dx 
\notag 
\end{align} 
are absolutely convergent. 
\item[(iii)] 
If $q>p$ and $h=0$, 
then for any $l=1,\dots,l_{0}$, or if $p > 0$, $q > 0$ and $h = 1$, then for any $l \in \mathbb{N}$, 
\begin{align} 
\big| e^{\pm i\lambda x^{p}} (\pm i\lambda px^{p-1})^{-1} (\pm L^{*})^{l-1} (x^{q-1} a_{h}(x) \chi_{\varepsilon}^{(k)}(x)) \big| \to 0 
\notag 
\end{align} 
as $x \to +0$ or $x \to \infty$. 
\item[(iv)] 
If $q>p$ and $h=0$, 
then for any $l=1,\dots,l_{0}$, or if $p > 0$, $q > 0$ and $h = 1$, then for any $l \in \mathbb{N}$, 
\begin{align} 
\int_{0}^{\infty} e^{\pm i\lambda x^{p}} x^{q-1} a_{h}(x) \chi_{\varepsilon}^{(k)}(x) dx 
= \int_{0}^{\infty} e^{\pm i\lambda x^{p}} (\pm L^{*})^{l} (x^{q-1} a_{h}(x) \chi_{\varepsilon}^{(k)}(x)) dx. 
\label{I_pq_pm_a_lambda} 
\end{align} 
\end{enumerate} 
\end{lem} 

\begin{proof} 
Since the lower side of the double sign $\pm$'s can be obtained as the conjugate of the upper one, 
we shall show the upper one. 

(i) 
By induction on $l \in \mathbb{Z}_{\geq 0}$. 
If $l=0$, 
then \eqref{L star_l_01} with $C_{0,0} = 1$ holds. 
Next if \eqref{L star_l_01} holds for $l-1$ with $l \geq 1$, 
then defining $C_{l,j}$ by \eqref{C_l_j_def01} and \eqref{C_l_j_def02}, we have 
\begin{align} 
&L^{*l} (x^{q-1} a_{h}(x) \chi_{\varepsilon}^{(k)}(x)) 
\notag \\ 
&= - \frac{1}{i\lambda} \frac{d}{dx} \frac{1}{px^{p-1}} \left( \frac{i}{\lambda p} \right)^{l-1} \sum_{j=0}^{l-1} C_{l-1,j} x^{q-1-p(l-1)+j} (a_{h}(x) \chi_{\varepsilon}^{(k)}(x))^{(j)} \notag \\ 
&= \left( \frac{i}{\lambda p} \right)^{l} \sum_{j=0}^{l-1} C_{l-1,j} \notag \\ 
&\hspace{0.75cm}\times \left\{ (q-pl+j) x^{q-1-pl+j} (a_{h}(x) \chi_{\varepsilon}^{(k)}(x))^{(j)} + x^{q-pl+j} (a_{h}(x) \chi_{\varepsilon}^{(k)}(x))^{(j+1)} \right\} \notag \\ 
&= \left( \frac{i}{\lambda p} \right)^{l} \Bigg[ \sum_{j=1}^{l-1} \big\{ (q-pl+j) C_{l-1,j} + C_{l-1,j-1} \big\} x^{q-1-pl+j} (a_{h}(x) \chi_{\varepsilon}^{(k)}(x))^{(j)} \Bigg. \notag \\ 
&\hspace{0.75cm}\Bigg. + (q-pl) C_{l-1,0} x^{q-1-pl} a_{h}(x) \chi_{\varepsilon}^{(k)}(x) + C_{l-1,l-1} x^{q-1-pl+l} (a_{h}(x) \chi_{\varepsilon}^{(k)}(x))^{(l)} \Bigg] \notag \\ 
&= \left( \frac{i}{\lambda p} \right)^{l} \sum_{j=0}^{l} C_{l,j} x^{q-1-pl+j} (a_{h}(x) \chi_{\varepsilon}^{(k)}(x))^{(j)}. \notag 
\end{align} 
Therefore \eqref{L star_l_01} with \eqref{C_l_0_02} holds. 

(ii) 
For $h=0,1$ and $l \in \mathbb{Z}_{\geq 0}$, put 
\begin{align} 
f_{h,l}(x) = e^{i\lambda x^{p}} L^{*l} (x^{q-1} a_{h}(x) \chi_{\varepsilon}^{(k)}(x)). 
\notag 
\end{align} 
Applying Leibniz's formula to \eqref{L star_l_01} with $a_{h} := a \psi_{h}$ and \eqref{A_tau_delta_def02}, 
for any $x \in (0,\infty)$, 
\begin{align} 
&|f_{h,l}(x)| 
= |L^{*l} (x^{q-1} a_{h}(x) \chi_{\varepsilon}^{(k)}(x))| \notag \\ 
&\leq (\lambda p)^{-l} \sum_{j=0}^{l} |C_{l,j}| |x|^{q-1-pl+j} \sum_{s+t+u=j} \frac{j!}{s!t!u!} |a|^{(\tau)}_{l} \langle x \rangle^{\tau + \delta s} |\psi_{h}^{(t)}(x)| |\chi_{\varepsilon}^{(k+u)}(x)|. 
\notag 
\end{align} 
Here since $\psi_{0} \equiv 1$, then $\langle x \rangle^{t} |\psi_{0}^{(t)}(x)| \equiv \delta_{t0}$ on $\mathbb{R}$, 
and since $\psi_{1} = 1 - \varphi$, then 
\begin{align} 
\langle x \rangle^{t} |\psi_{1}^{(t)}(x)| = 
\begin{cases} 
0, & \text{if}~|x| \leq r_{0}, \\ 
\langle x \rangle^{t} |\delta_{t0} - \varphi^{(t)}(x)|, &  \text{if}~r_{0} < |x| < r_{1}, \\ 
\delta_{t0}, & \text{if}~|x| \geq r_{1}, 
\end{cases} 
\notag 
\end{align} 
where $\delta_{t0}$ is the Kronecker's delta. 
Hence we can define 
\begin{align} 
|\psi_{h}|_{l} 
:= \max_{t=0,\dots,l} \sup_{x \in \mathbb{R}} \langle x \rangle^{t} |\psi_{h}^{(t)}(x)| \in [1,\infty) 
\label{|psi_{h}|_{l,2}} 
\end{align} 
for $h=0,1$. Then since $|\psi_{h}^{(t)}(x)| \leq \langle x \rangle^{-t} |\psi_{h}|_{l}$ with $-1 \leq \delta$, for any $x \in (0,\infty)$, 
\begin{align} 
&|f_{h,l}(x)| 
= |L^{*l} (x^{q-1} a_{h}(x) \chi_{\varepsilon}^{(k)}(x))| \notag \\ 
&\leq (\lambda p)^{-l} \sum_{j=0}^{l} |C_{l,j}| |x|^{q-1-pl+j} 
\sum_{s+t+u=j} \frac{j!}{s!t!u!} |a|^{(\tau)}_{l} \langle x \rangle^{\tau + \delta (s+t)} |\psi_{h}|_{l} |\chi_{\varepsilon}^{(k+u)}(x)|. 
\label{|f_{h,l}(x)|02} 
\end{align} 
(In the case of $h=1$, this inequality will be used as \eqref{L^*l_est00} in the proof of Theorem \ref{Lax02} (ii).) 

Here if $|x| \geq 1$, since $|x| \leq 2^{1/2} \langle x \rangle$ and $2^{-1/2} \langle x \rangle \leq |x|$, 
then 
for any $t \in \mathbb{R}$, 
\begin{align} 
|x|^{t} \leq 2^{|t|/2} \langle x \rangle^{t}. 
\label{|x|_angle_x} 
\end{align} 
Since $\chi \in \mathcal{S}(\mathbb{R})$, there exists a positive constant $\tilde{C}_{k+u}$ such that for any $m \in \mathbb{Z}_{\geq 0}$, 
\begin{align} 
|\chi_{\varepsilon}^{(k+u)}(x)| 
= |\partial_{y}^{k+u} \chi (\varepsilon x) \varepsilon^{k+u}| 
\leq \tilde{C}_{k+u} \langle \varepsilon x \rangle^{-m} 
\leq \tilde{C}_{k+u} \varepsilon^{-m} \langle x \rangle^{-m}. 
\label{chi_est00} 
\end{align} 
Applying \eqref{|x|_angle_x}, \eqref{chi_est00} to \eqref{|f_{h,l}(x)|02}, for any $x \in [1,\infty)$ and for any $m \in \mathbb{Z}_{\geq 0}$, we have 
\begin{align} 
|f_{h,l}(x)| 
&\leq (\lambda p)^{-l} \sum_{j=0}^{l} |C_{l,j}| \cdot 2^{|q-1-pl+j|/2} \langle x \rangle^{q-1-pl+j} \notag \\ 
&\hspace{0.75cm}\times \sum_{s+t+u=j} \frac{j!}{s!t!u!} |a|^{(\tau)}_{l} \langle x \rangle^{\tau + \delta j} \langle x \rangle^{u} |\psi_{h}|_{l} \tilde{C}_{k+u} \varepsilon^{-m} \langle x \rangle^{-m} \notag \\ 
&\leq C^{(k)}_{l,\varepsilon} \lambda^{-l} |a|^{(\tau)}_{l} \langle x \rangle^{q-1+\tau -(p-1-\delta)l+l-m}, 
\label{|f_{h,l}(x)|03} 
\end{align} 
where $C^{(k)}_{l,\varepsilon}$ is a positive constant dependent on $\varepsilon$. 
Since $f(x) = O(x^{\beta})~(x \to \infty)$ with $\beta = q-1+\tau -(p-1-\delta)l+l-m$ for any $m \in \mathbb{Z}_{\geq 0}$, 
taking $m$ such that $\beta < -1$, then $\int_{1}^{\infty} f_{h,l}(x) dx$ is absolutely convergent. 

When $q>p$ and $h=0$, since $a_{0} = a$, by \eqref{|f_{h,l}(x)|02}, for any $x \in (0,1)$, 
\begin{align} 
|f_{0,l}(x)| \leq C_{l}|x|^{q-1-pl}, 
\label{f_m_orger_alpha} 
\end{align} 
where $C_{l}$ is a positive constant. 
Here let $l_{0}:= [q/p)$, since $(q/p)-1 \leq l_{0} < q/p$, then $0 < q-pl_{0} \leq p$. 
Hence for any $l=0,\dots,l_{0}$, since $f_{0,l}(x) = O(x^{\alpha})~(x \to +0)$ with $\alpha = q-1-pl \geq q-1-pl_{0} > -1$, then $\int_{0}^{1} f_{0,l}(x) dx$ is absolutely convergent. 

When $p > 0$, $q > 0$ and $h = 1$, since $a_{1} = a(1 - \varphi) \equiv 0$ on $(0,1]$, then $f_{1,l} \equiv 0$ on $(0,1]$ for $l \in \mathbb{Z}_{\geq 0}$. 
Then $\int_{0}^{1} f_{1,l}(x) dx$ is absolutely convergent. 

Therefore $\int_{0}^{\infty} f_{h,l}(x) dx = \int_{0}^{1} f_{h,l}(x) dx + \int_{1}^{\infty} f_{h,l}(x) dx$ is absolutely convergent. 

(iii) 
For $h=0,1$ and $l \in \mathbb{N}$, put 
\begin{align} 
g_{h,l-1}(x) = (i\lambda px^{p-1})^{-1} f_{h,l-1}(x). 
\notag 
\end{align} 
By \eqref{|x|_angle_x} and \eqref{|f_{h,l}(x)|03}, for any $x \in (1,\infty)$ and for any $m \in \mathbb{Z}_{\geq 0}$, 
\begin{align} 
|g_{h,l-1}(x)| 
\leq C_{1} \langle x \rangle^{-(p-1)} C_{2} \langle x \rangle^{q-1+\tau -(p-1-\delta)(l-1)+(l-1)-m} 
\leq C_{1} C_{2} \langle x \rangle^{q+\tau+l-m}, 
\notag 
\end{align} 
where $C_{1}$ and $C_{2}$ are positive constants. 
Here taking $m \in \mathbb{Z}_{\geq 0}$ such that $q+\tau+l-m<0$, then $|g_{h,l-1}(x)| \to 0$ as $x \to \infty$. 

When $q>p$ and $h=0$, by \eqref{f_m_orger_alpha}, for any $x \in (0,1)$, 
\begin{align} 
|g_{0,l-1}(x)| 
\leq C_{3}|x|^{-(p-1)} C_{l-1} |x|^{q-1-p(l-1)} 
= C_{3} C_{l-1} |x|^{q-pl}, 
\notag 
\end{align} 
where $C_{3}$ and $C_{l-1}$ are positive constants. 
Here for any $l=1,\dots,l_{0}$, since $q-pl \geq q-pl_{0} > 0$, then $|g_{0,l-1}(x)| \to 0$ as $x \to +0$. 

When $p > 0$, $q > 0$ and $h = 1$, by (ii), since $f_{1,l-1} \equiv 0$, then $g_{1,l-1} \equiv 0$ on $(0,1]$ for $l \in \mathbb{N}$. 
Therefore $|g_{1,l-1}(x)| \to 0$ as $x \to +0$. 

(iv) 
By induction on $l \in \mathbb{N}$. 
If $l=1$, since $L(e^{i\lambda x^{p}}) = e^{i\lambda x^{p}}$ when $x \ne 0$, 
by integration by parts with (ii) and (iii), 
then the following holds as $l=1$. 
Thus \eqref{I_pq_pm_a_lambda} holds for $l=1$. 
Next if \eqref{I_pq_pm_a_lambda} holds for $l-1$ with $l \geq 2$, 
then we have 
\begin{align} 
&\int_{0}^{\infty} e^{i\lambda x^{p}} x^{q-1} a_{h}(x) \chi_{\varepsilon}^{(k)}(x) dx 
= \int_{0}^{\infty} e^{i\lambda x^{p}} L^{*l-1} (x^{q-1} a_{h}(x) \chi_{\varepsilon}^{(k)}(x)) dx \notag \\ 
&= \lim_{\substack{u \to +0\\v \to \infty}} \int_{u}^{v} \frac{1}{px^{p-1}} \frac{1}{i\lambda} \frac{d}{dx}(e^{i\lambda x^{p}}) L^{*l-1} (x^{q-1} a_{h}(x) \chi_{\varepsilon}^{(k)}(x)) dx \notag \\ 
&= \lim_{\substack{u \to +0\\v \to \infty}} \bigg\{ \Big[ e^{i\lambda x^{p}} (i\lambda px^{p-1})^{-1} L^{*l-1} (x^{q-1} a_{h}(x) \chi_{\varepsilon}^{(k)}(x)) \Big]_{u}^{v} \bigg. \notag \\ 
&\hspace{0.65cm}\bigg. + \int_{u}^{v} e^{i\lambda x^{p}} L^{*l} (x^{q-1} a_{h}(x) \chi_{\varepsilon}^{(k)}(x)) dx \bigg\} \notag \\ 
&= \int_{0}^{\infty} e^{i\lambda x^{p}} L^{*l} (x^{q-1} a_{h}(x) \chi_{\varepsilon}^{(k)}(x)) dx. 
\notag 
\end{align} 
This completes the proof. 
\end{proof} 

By Lemma \ref{L_star_l}, 
we obtain the following theorem: 
\begin{thm}[\cite{Nagano-Miyazaki04}] 
\label{Lax02} 
Assume that 
$\lambda > 0$, $p > 0$, $q > 0$, $a \in \mathcal{A}^{\tau}_{\delta}(\mathbb{R})$ and $\chi_{\varepsilon}(x) := \chi(\varepsilon x)$ for $x \in \mathbb{R}$ where $\chi \in \mathcal{S}(\mathbb{R})$ with $\chi(0) = 1$ and $0 < \varepsilon <1 $. 
Let 
$\varphi \in C^{\infty}_{0}(\mathbb{R})$ be a cut-off function such that $\varphi \equiv 1$ on $|x| \leq r_{0}$ with $r_{0} \geq 1$ and $\varphi \equiv 0$ on $|x| \geq r_{1} > r_{0}$, 
and $\psi := 1 - \varphi$, and let $L^{*} := - \frac{1}{i\lambda} \frac{d}{dx}  \frac{1}{px^{p-1}}$ and $l_{0}:=[q/p)$. 
Then the following hold: 
\begin{enumerate} 
\item[(i)] 
For each $k \in \mathbb{Z}_{\geq 0}$, 
there exist the following limit of improper integrals independent of $\chi$, and the following hold: 
\begin{align} 
\lim_{\varepsilon \to +0} \int_{0}^{\infty} e^{\pm i\lambda x^{p}} x^{q-1} a(x) \varphi(x) \chi_{\varepsilon}^{(k)}(x) dx 
= \delta_{k0} \int_{0}^{\infty} e^{\pm i\lambda x^{p}} x^{q-1} a(x) \varphi(x) dx, 
\notag 
\end{align} 
where $\delta_{k0}$ is the Kronecker's delta. 
\item[(ii)] 
For each $k \in \mathbb{Z}_{\geq 0}$, 
there exist the following limit of improper integrals independent of $\chi$, 
and for any $l \in \mathbb{N}$ such that $l \geq l_{p,q}$, the following hold: 
\begin{align} 
&\lim_{\varepsilon \to +0} \int_{0}^{\infty} e^{\pm i\lambda x^{p}} x^{q-1} a(x) \psi(x) \chi_{\varepsilon}^{(k)}(x) dx 
= \delta_{k0} \int_{0}^{\infty} e^{\pm i\lambda x^{p}} (\pm L^{*})^{l} (x^{q-1} a(x) \psi(x)) dx, 
\notag 
\end{align} 
where 
\begin{align} 
l_{p,q} := \left[ \frac{(q+\tau)^{+}}{p-1-\delta} \right] +1. 
\label{m_def} 
\end{align} 
In particular, if $k=0$, there exist the following oscillatory integrals, and for any $l \in \mathbb{N}$ such that $l \geq l_{p,q}$, the following hold: 
\begin{align} 
\tilde{I}_{p,q}^{\pm}[a \psi](\lambda) 
:&= Os\text{-}\int_{0}^{\infty} e^{\pm i\lambda x^{p}} x^{q-1} a(x) \psi(x) dx \notag \\ 
&= \int_{0}^{\infty} e^{\pm i\lambda x^{p}} (\pm L^{*})^{l} (x^{q-1} a(x) \psi(x)) dx, 
\label{existence_of_oscillatory integrals_a_psi_01} 
\end{align} 
and there exists a positive constant $\tilde{C}^{(0)}_{l}$ 
such that for any $\lambda > 0$, 
\begin{align} 
| \tilde{I}_{p,q}^{\pm}[a \psi](\lambda) | 
\leq \tilde{C}^{(0)}_{l} |a|^{(\tau)}_{l} \lambda^{-l}, 
\label{existence_of_oscillatory integrals_a_psi_02} 
\end{align} 
where $|a|^{(\tau)}_{l}$ is defined by \eqref{|a|^(tau)_l}. 
\item[(iii)] 
If $k \ne 0$, then 
\begin{align} 
\lim_{\varepsilon \to +0} \int_{0}^{\infty} e^{\pm i\lambda x^{p}} x^{q-1} a(x) \chi_{\varepsilon}^{(k)}(x) dx 
= 0. 
\label{when_k=0} 
\end{align} 
\item[(iv)] 
There exist the following oscillatory integrals, and the following hold: 
\begin{align} 
\tilde{I}_{p,q}^{\pm}[a](\lambda) 
:= Os\text{-}\int_{0}^{\infty} e^{\pm i\lambda x^{p}} x^{q-1} a(x) dx 
= \tilde{I}_{p,q}^{\pm}[a \varphi](\lambda) + \tilde{I}_{p,q}^{\pm}[a \psi](\lambda). 
\label{existence_of_oscillatory integrals_new} 
\end{align} 
Then for each $l \in \mathbb{N}$ such that $l \geq l_{p,q}$, there exists a positive constant $C_{l}$ 
such that for any $\lambda \geq 1$, 
\begin{align} 
|\tilde{I}_{p,q}^{\pm}[a](\lambda)| \leq C_{l} |a|^{(\tau)}_{l}. 
\label{tilde_I_m_est} 
\end{align} 
\item[(v)] 
If $q > p$, 
then there exists a positive constant $C_{p,q}$ such that for any $\lambda \geq 1$, 
\begin{align} 
| \tilde{I}_{p,q}^{\pm}[a](\lambda) | 
\leq C_{p,q} |a|^{(\tau)}_{l_{0}+l_{p,q}} \lambda^{-\frac{q-p}{p}}. 
\label{tilde_I_m_est02} 
\end{align} 
\end{enumerate} 
\end{thm} 

\begin{proof} 
Since the lower side of the double sign $\pm$'s can be obtained as the conjugate of the upper one, we shall show the upper one. 

(i) 
Put $f_{k}(x) = e^{i\lambda x^{p}} x^{q-1} a(x) \varphi(x) \chi_{\varepsilon}^{(k)}(x)$ for $k \in \mathbb{Z}_{\geq 0}$. 
Since $f_{k}$ is continuous on $(0,\infty)$ and $f_{k} \equiv 0$ on $[r_{1},\infty)$, 
then $f_{k}$ is integrable on $[u,\infty)$ for any $u \in (0,\infty)$ and $f_{k}(x) = O(x^{\alpha})~(x \to +0)$ with $\alpha =q-1>-1$. 
Thus the improper integral 
\begin{align} 
\int_{0}^{\infty} e^{i\lambda x^{p}} x^{q-1} a(x) \varphi(x) \chi_{\varepsilon}^{(k)}(x) dx 
\notag 
\end{align} 
is absolutely convergent for any $k \in \mathbb{Z}_{\geq 0}$. 
By Proposition \ref{chi epsilon} (ii) in \S \ref{Preliminary}, 
there exists a positive constant $C_{0}$ independent of $0 < \varepsilon <1 $ such that for any $x \in (0,\infty)$, 
\begin{align} 
|e^{i\lambda x^{p}} x^{q-1} a(x) \varphi(x) \chi_{\varepsilon}^{(k)}(x)| 
\leq C_{0} |a|^{(\tau)}_{0} x^{q-1} \langle x \rangle^{\tau} |\varphi(x)| =: M(x). 
\label{M_est} 
\end{align} 
Since $\int_{0}^{\infty} M(x) dx$ is absolutely convergent independent of $\chi$, 
by using Lebesgue's dominated convergence theorem with Proposition \ref{chi epsilon} (i) and (iii) in \S \ref{Preliminary}, 
for each $k \in \mathbb{Z}_{\geq 0}$, 
there exists the following limit of the improper integral independent of $\chi$, and the following holds: 
\begin{align} 
\lim_{\varepsilon \to +0} \int_{0}^{\infty} e^{i\lambda x^{p}} x^{q-1} a(x) \varphi(x) \chi_{\varepsilon}^{(k)}(x) dx 
= \delta_{k0} \int_{0}^{\infty} e^{i\lambda x^{p}} x^{q-1} a(x) \varphi(x) dx. 
\notag 
\end{align} 

(ii) By Lemma \ref{L_star_l} (ii) when $h=1$ and $l=0$, the improper integral 
\begin{align} 
\int_{0}^{\infty} e^{i\lambda x^{p}} x^{q-1} a(x) \psi(x) \chi_{\varepsilon}^{(k)}(x) dx 
\label{tilde_I_phi_a_psi02} 
\end{align} 
is absolutely convergent for any $k \in \mathbb{Z}_{\geq 0}$. 
In order to apply Lebesgue's dominated convergence theorem, 
we shall show \eqref{tilde_I_phi_a_psi02} is bounded independent of $\chi$ for any $k \in \mathbb{Z}_{\geq 0}$. 
By Lemma \ref{L_star_l} (iv) when $h=1$, 
for any $l \in \mathbb{N}$, 
\begin{align} 
\int_{0}^{\infty} e^{i\lambda x^{p}} x^{q-1} a(x) \psi(x) \chi_{\varepsilon}^{(k)}(x) dx 
= \int_{0}^{\infty} e^{i\lambda x^{p}} L^{*l} (x^{q-1} a(x) \psi(x) \chi_{\varepsilon}^{(k)}(x)) dx. 
\label{L star} 
\end{align} 
By Lemma \ref{L_star_l} (i) with Proposition \ref{chi epsilon} (ii) in \S \ref{Preliminary}, 
this means that the order of integrand descends to be integrable in the sense of Lebesgue by $L^{*l}$ with sufficiently large number $l \gg 0$. 
We shall show this. 
By \eqref{|f_{h,l}(x)|02} when $h=1$ with \eqref{C_l_0_02}, 
for each $l \in \mathbb{Z}_{\geq 0}$, there exist real constants $(C_{l,0},\dots,C_{l,l}) \ne (0,\dots,0)$ such that 
for any $x \in (0,\infty)$, 
\begin{align} 
&|L^{*l} (x^{q-1} a(x) \psi(x) \chi_{\varepsilon}^{(k)}(x))| \notag \\ 
&\leq (\lambda p)^{-l} \sum_{j=0}^{l} |C_{l,j}| |x|^{q-1-pl+j} \sum_{s+t+u=j} \frac{j!}{s!t!u!} |a|^{(\tau)}_{l} \langle x \rangle^{\tau + \delta (s+t)} |\psi|_{l} |\chi_{\varepsilon}^{(k+u)}(x)|, 
\label{L^*l_est00} 
\end{align} 
where $|\psi|_{l} := \max_{t=0,\dots,l} \sup_{x \in \mathbb{R}} \langle x \rangle^{t} |\psi^{(t)}(x)|$. 
Here since $\psi \equiv 0$ on $|x| \leq 1$, then $\mathrm{supp} \psi \cap (0,\infty) \subset [1,\infty)$. 
Hence for any $x \in \mathrm{supp} \psi \cap (0,\infty)$, since $|x| \geq 1$, by \eqref{|x|_angle_x}, 
\begin{align} 
|x|^{q-1-pl+j} \leq 2^{|q-1-pl+j|/2} \langle x \rangle^{q-1-pl+j}. 
\label{|x|_angle_x02}
\end{align} 
Also by Proposition \ref{chi epsilon} (ii) in \S \ref{Preliminary} with $-1 \leq \delta$, for each $k \in \mathbb{Z}_{\geq 0}$ and for each $u=0,\dots,l$, 
there exists a positive constant $C_{k+u}$ independent of $0 < \varepsilon <1 $ such that for any $x \in [0,\infty)$, 
\begin{align} 
|\chi_{\varepsilon}^{(k+u)}(x)| 
\leq C_{k+u} \langle x \rangle^{\delta u}. 
\label{partial_x_k+w_chi_varepsilon_x} 
\end{align} 
According to \eqref{L^*l_est00} with \eqref{|x|_angle_x02} and \eqref{partial_x_k+w_chi_varepsilon_x}, 
for any $x \in (0,\infty)$, 
\begin{align} 
|L^{*l} (x^{q-1} a(x) \psi(x) \chi_{\varepsilon}^{(k)}(x))| 
&\leq (\lambda p)^{-l} \sum_{j=0}^{l} |C_{l,j}| \cdot 2^{|q-1-pl+j|/2} \langle x \rangle^{q-1-pl+j} \notag \\ 
&\hspace{0.75cm}\times \sum_{s+t+u=j} \frac{j!}{s!t!u!} |a|^{(\tau)}_{l} \langle x \rangle^{\tau + \delta (s+t)} |\psi|_{l} C_{k+u} \langle x \rangle^{\delta u} \notag \\ 
&\leq C^{(k)}_{l} \lambda^{-l} |a|^{(\tau)}_{l} \langle x \rangle^{\beta} =: M_{k}(x), 
\label{estimate of L star} 
\end{align} 
where $C^{(k)}_{l}$ is a positive constant independent of $\varepsilon$ 
and 
\begin{align} 
\beta = (q+\tau)^{+} -(p-1-\delta)l-1. 
\notag 
\end{align} 
Here define $l_{p,q}$ by $\eqref{m_def}$, since $x-[x]-1<0$ for any $x \in \mathbb{R}$, then for any $l \in \mathbb{N}$ such that $l \geq l_{p,q}$, 
\begin{align} 
\beta 
&\leq (p-1-\delta) \left\{ \frac{(q+\tau)^{+}}{p-1-\delta} - \left[ \frac{(q+\tau)^{+}}{p-1-\delta} \right] -1 \right\} -1 < -1. 
\notag 
\end{align} 
Thus for any $k \in \mathbb{Z}_{\geq 0}$, $\int_{0}^{\infty} M_{k}(x) dx$ is absolutely convergent independent of $\chi$. 

Therefore by applying Lebesgue's dominated convergence theorem to the right hand side in \eqref{L star} as $\varepsilon \to +0$, and using Proposition \ref{chi epsilon} (i) and (iii) in \S \ref{Preliminary} to \eqref{L star_l_01}, 
then for each $k \in \mathbb{Z}_{\geq 0}$, 
there exists the following limit of the improper integral independent of $\chi$, and for any $l \in \mathbb{N}$ such that $l \geq l_{p,q}$, the following holds: 
\begin{align} 
&\lim_{\varepsilon \to +0} \int_{0}^{\infty} e^{i\lambda x^{p}} x^{q-1} a(x) \psi(x) \chi_{\varepsilon}^{(k)}(x) dx 
= \delta_{k0} \int_{0}^{\infty} e^{i\lambda x^{p}} L^{*l} (x^{q-1} a(x) \psi(x) ) dx. 
\notag 
\end{align} 
If $k=0$, there exists $\tilde{I}_{p,q}^{+}[a \psi](\lambda)$ with \eqref{existence_of_oscillatory integrals_a_psi_01}. 
Then as $\varepsilon \to +0$ in \eqref{estimate of L star}, \eqref{existence_of_oscillatory integrals_a_psi_02} holds for any $\lambda > 0$. 

(iii) 
If $k \ne 0$, since $\varphi + \psi \equiv 1$, by (i) and (ii), \eqref{when_k=0} holds. 

(iv) 
If $k = 0$, since $\varphi + \psi \equiv 1$, by (i) and (ii), there exists $\tilde{I}_{p,q}^{+}[a](\lambda)$ with \eqref{existence_of_oscillatory integrals_new}. 
Then as $\varepsilon \to +0$ in \eqref{M_est} when $k=0$ with \eqref{existence_of_oscillatory integrals_a_psi_02}, 
\eqref{tilde_I_m_est} holds for any $\lambda \geq 1$. 

(v) 
When $q > p$, let $l_{0} = [q/p)$, then $(q/p)-1 \leq l_{0} < q/p$. 
Because $q-pl_{0}+j >0$ and $a^{(j)} \in \mathcal{A}^{\tau +\delta j}_{\delta}(\mathbb{R})$ for $j=0,\dots,l_{0}$, 
by Lemma \ref{L_star_l} (iv) and (i) when $h=0$ and $k=0$, and by (iii) and (iv), we have 
\begin{align} 
\tilde{I}_{p,q}^{+}[a](\lambda) 
&= \lim_{\varepsilon \to +0} \int_{0}^{\infty} e^{i\lambda x^{p}} \left( \frac{i}{\lambda p} \right)^{l_{0}} \sum_{j=0}^{l_{0}} C_{l_{0},j} x^{q-1-pl_{0}+j} (a(x) \chi_{\varepsilon}(x))^{(j)} dx \notag \\ 
&= \left( \frac{i}{\lambda p} \right)^{l_{0}} \sum_{j=0}^{l_{0}} C_{l_{0},j} \tilde{I}_{p,q-pl_{0}+j}^{+}[a^{(j)}](\lambda), 
\label{q>p_tilde_I_pq_+_a(lambda)} 
\end{align} 
where $(C_{l_{0},0},\dots, C_{l_{0},l_{0}}) \ne (0,\dots,0)$. 
Here by \eqref{m_def}, for any $j=0,\dots,l_{0}$, let 
\begin{align} 
w_{j} := \left[ \frac{(q-pl_{0}+j+\tau +\delta j)^{+}}{p-1-\delta} \right] +1 
\label{m_j_def} 
\end{align} 
Then $w_{j} \leq w_{l_{0}}$. 
Applying \eqref{tilde_I_m_est}, \eqref{A_tau_delta_def03} in \S \ref{Preliminary} to \eqref{q>p_tilde_I_pq_+_a(lambda)}, 
for any $\lambda \geq 1$, 
\begin{align} 
|\tilde{I}_{p,q}^{+}[a](\lambda)| 
&\leq (\lambda p)^{-l_{0}} \sum_{j=0}^{l_{0}} |C_{l_{0},j}| C_{w_{j}} |a|^{(\tau)}_{j+w_{j}} 
\leq C_{p,q} |a|^{(\tau)}_{l_{0}+w_{l_{0}}} \lambda^{-\frac{q}{p}+1}, 
\label{|tilde_I_p_q_+_a_(lambda)|} 
\end{align} 
where $C_{w_{j}}$ and $C_{p,q}$ are positive constants. 
Here by \eqref{m_j_def} and \eqref{m_def}, we see 
\begin{align} 
w_{l_{0}} 
&= \left[ \frac{\{ q+\tau - (p-1-\delta) l_{0} \}^{+}}{p-1-\delta} \right] +1 
\leq 
l_{p,q}. 
\label{l_{0} + w_{l_{0}}} 
\end{align} 

Therefore by \eqref{|tilde_I_p_q_+_a_(lambda)|} and \eqref{l_{0} + w_{l_{0}}}, 
\eqref{tilde_I_m_est02} holds 
for $\lambda \geq 1$. 
\end{proof} 

If $p = m \in \mathbb{N}$, $q=1$ and $k=0$, by change of variable, an oscillatory integral on another half-line $(-\infty,0]$ can be reduced to the one on half-line $[0,\infty)$ as follows: 
\begin{cor} 
\label{Os_m} 
Assume that 
$\lambda > 0$ and $a \in \mathcal{A}^{\tau}_{\delta}(\mathbb{R})$. 
Let 
$m \in \mathbb{N}$. 
Then there exist the following oscillatory integral and the following hold: 
\begin{align} 
\tilde{I}_{m,1}^{\pm \pm^{m}}[a(-y)](\lambda) 
:&= Os\text{-}\int_{0}^{\infty} e^{\pm (-1)^{m}i\lambda y^{m}} a(-y) dy \notag \\ 
&= \lim_{\varepsilon \to +0} \int_{-\infty}^{0} e^{\pm i \lambda x^{m}} a (x) \chi (\varepsilon x) dx 
=: Os\text{-}\int_{-\infty}^{0} e^{\pm i\lambda x^{m}} a(x) dx, 
\label{Os_int_-infty_0_e^pm_i_lambda_x^m_a(x)_dx_def01} 
\end{align} 
where $\chi \in \mathcal{S}(\mathbb{R})$ with $\chi(0) = 1$. Then we have 
\begin{align} 
\tilde{J}_{m,1}^{\pm}[a](\lambda) 
:&= Os\text{-}\int_{-\infty}^{\infty} e^{\pm i\lambda x^{m}} a(x) dx \notag \\ 
&= Os\text{-}\int_{0}^{\infty} e^{\pm i\lambda x^{m}} a(x) dx + Os\text{-}\int_{-\infty}^{0} e^{\pm i\lambda x^{m}} a(x) dx. 
\label{Os_int_-infty_infty_e^pm_i_lambda_x^m_a(x)_dx_def01} 
\end{align} 
\end{cor} 

\begin{proof} 
Since $a(-y) \in \mathcal{A}^{\tau}_{\delta}(\mathbb{R})$ and $\chi(-y) \in \mathcal{S}(\mathbb{R})$, by Theorem \ref{Lax02} (iv) when $p=m \in \mathbb{N}$ and $q=1$, 
there exist oscillatory integrals $\tilde{I}_{m,1}^{\pm \pm^{m}}[a(-y)](\lambda)$, and by change of variable $y = -x$, we obtain \eqref{Os_int_-infty_0_e^pm_i_lambda_x^m_a(x)_dx_def01}. Moreover by \eqref{oscillatory_integral_def02}, we obtain \eqref{Os_int_-infty_infty_e^pm_i_lambda_x^m_a(x)_dx_def01}. 
\end{proof} 

\section{Generalized Fresnel Integrals} 
\label{section_Generalized_Fresnel_Integrals} 

In this section, 
we consider a generalization of Fresnel integrals. 

\begin{lem}[\cite{Nagano01},\cite{Nagano-Miyazaki03},\cite{Nagano-Miyazaki04}] 
\label{Generalized Fresnel integrals} 
Assume that $p > q > 0$. 
Then the following hold: 
\begin{align} 
I_{p,q}^{\pm} 
:= \int_{0}^{\infty} e^{\pm ix^{p}} x^{q-1} dx 
= p^{-1} e^{\pm i\frac{\pi}{2} \frac{q}{p}} \varGamma \left( \frac{q}{p} \right). 
\label{I_pq} 
\end{align} 
\end{lem} 

\begin{proof} 
Suppose $p > q > 0$. 
Consider the following $C^{1}$-curves in $\mathbb{C}$ : 
\begin{align} 
C_{1} &: z=x~(0 < r  \leq x \leq R), \notag \\ 
C_{2} &: z=R e^{i\theta}~(0 \leq \theta \leq \pi /2p), \notag \\ 
C_{3} &: z=-s e^{i(\pi /2p)}~(-R  \leq s \leq -r), \notag \\ 
C_{4} &: z=r e^{-i\tau}~(-\pi /2p \leq \tau \leq 0), \notag 
\end{align} 
and a domain $D$ with the anticlockwise oriented boundary $\sum_{j=1}^{4} C_{j}$. 
Since $e^{iz^{p}} z^{q-1}$ is holomorphic in $D$ for $p > q > 0$, 
by Cauchy's integral theorem, 
\begin{align} 
0 
= \int_{\sum_{j=1}^{4} C_{j}} e^{iz^{p}} z^{q-1} dz 
= \sum_{j=1}^{4} \int_{C_{j}} e^{iz^{p}} z^{q-1} dz. 
\label{Cauchy} 
\end{align} 

As to $\int_{C_{2}} e^{iz^{p}} z^{q-1} dz$, 
by Jordan's inequality:$(2p/\pi) \theta < \sin (p\theta)$ for $0 < \theta < \pi/2p$, we have 
\begin{align} 
&\left| \int_{C_{2}} e^{iz^{p}} z^{q-1} dz \right| 
= \left| \int_{0}^{\frac{\pi}{2p}} e^{i(Re^{i\theta})^{p}} (Re^{i\theta})^{q-1} Rie^{i\theta} d\theta \right| 
\leq R^{q} \int_{0}^{\frac{\pi}{2p}} e^{-R^{p} \sin (p\theta)} d\theta \notag \\ 
&< R^{q} \int_{0}^{\frac{\pi}{2p}} e^{-R^{p} \frac{2p}{\pi} \theta} d\theta 
= R^{q} \left[ \frac{1}{-R^{p} \frac{2p}{\pi}} e^{-R^{p} \frac{2p}{\pi} \theta} \right]_{0}^{\frac{\pi}{2p}} 
= \frac{\pi}{2p} \frac{1 - e^{-R^{p}}}{R^{p-q}} \to 0 
\label{C_2_integral} 
\end{align} 
as $R \to \infty$. 
As to $\int_{C_{4}} e^{iz^{p}} z^{q-1} dz$, 
by change of variable $\tau = -\theta$, we have 
\begin{align} 
&\left| \int_{C_{4}} e^{iz^{p}} z^{q-1} dz \right| 
= \left| \int_{-\frac{\pi}{2p}}^{0} e^{i(r e^{-i\tau})^{p}} (r e^{-i\tau})^{q-1} (-r ie^{-i\tau}) d\tau \right| \notag \\ 
&= \left| \int_{\frac{\pi}{2p}}^{0} e^{i(r e^{i\theta})^{p}} (r e^{i\theta})^{q-1} (-r ie^{i\theta}) (-d\theta) \right| 
\leq {r}^{q} \int_{0}^{\frac{\pi}{2p}} e^{-{r}^{p} \sin (p\theta)} d\theta 
< \frac{\pi}{2p} \frac{1 - e^{-r^{p}}}{r^{p-q}} \to 0 
\label{C_4_integral} 
\end{align} 
as $r \to +0$. 
Here we used l'H\^{o}pital's rule as follows, 
\begin{align} 
\lim_{r \to +0} \frac{1 - e^{-r^{p}}}{r^{p-q}} 
= \lim_{r \to +0} \frac{pr^{p-1} e^{-r^{p}}}{(p-q) r^{p-q-1}} 
= \lim_{r \to +0} \frac{pr^{q} e^{-r^{p}}}{p-q} 
= 0. 
\notag 
\end{align} 

Next $\int_{0}^{1} e^{ix^{p}} x^{q-1} dx$ is absolutely convergent. 
On the other hand, using integration by parts, we have 
\begin{align} 
\int_{1}^{v} e^{ix^{p}} x^{q-1} dx 
&= \int_{1}^{v} \frac{1}{ip} \frac{d}{dx} (e^{ix^{p}}) x^{q-p} dx \notag \\ 
&= \frac{1}{ip} \left\{ \Big[ e^{ix^{p}} x^{q-p} \Big] _{1}^{v} - (q-p) \int_{1}^{v} e^{ix^{p}} x^{q-p-1} dx \right\}. 
\notag 
\end{align} 
Here since $|e^{ix^{p}} x^{q-p}| \to 0$ as $x \to \infty$ and $\int_{1}^{\infty} e^{ix^{p}} x^{q-p-1} dx$ is absolutely convergent, the following integral is also absolutely convergent, and the following holds:  
\begin{align} 
\int_{1}^{\infty} e^{ix^{p}} x^{q-1} dx 
&= \frac{1}{ip} \left\{ -e^{i} - (q-p) \int_{1}^{\infty} e^{ix^{p}} x^{q-p-1} dx \right\}. 
\label{int_0_infty_e_ixp_x_q-1} 
\end{align} 
(This equality will be used at \eqref{Os_1_infty} in the proof of Theorem \ref{th01} (i).) 

Therefore the following integral is also absolutely convergent, 
and by \eqref{Cauchy}, \eqref{C_2_integral}, \eqref{C_4_integral}, and change of variables $s = -t^{1/p}$, 
the following holds:
\begin{align} 
I_{p,q}^{+} 
:&= \int_{0}^{\infty} e^{ix^{p}} x^{q-1} dx 
= \lim_{\substack{r \to +0\\ R \to \infty}} \int_{C_{1}} e^{iz^{p}} z^{q-1} dz 
= -\lim_{\substack{r \to +0\\ R \to \infty}} \int_{C_{3}} e^{iz^{p}} z^{q-1} dz \notag \\ 
&= - \lim_{\substack{r \to +0\\ R \to \infty}} \int_{-R}^{-r} e^{i(-se^{i\frac{\pi}{2p}})^{p}} (-se^{i\frac{\pi}{2p}})^{q-1} (-e^{i\frac{\pi}{2p}} ds) \notag \\ 
&= -\lim_{\substack{r \to +0\\ R \to \infty}} \int_{R^{p}}^{r^{p}} e^{i(t^{\frac{1}{p}}e^{i\frac{\pi}{2p}})^{p}} (t^{\frac{1}{p}}e^{i\frac{\pi}{2p}})^{q-1} (-e^{i\frac{\pi}{2p}}) (-p^{-1} t^{\frac{1}{p}-1} dt) \notag \\ 
&= p^{-1} e^{i\frac{\pi}{2} \frac{q}{p}} \int_{0}^{\infty} e^{-t} t^{\frac{q}{p}-1} dt 
= p^{-1} e^{i\frac{\pi}{2} \frac{q}{p}} \varGamma \left( \frac{q}{p} \right). 
\notag 
\end{align} 
\end{proof} 

When $q \geq p > 0$, 
we can make a sense of \eqref{I_pq} as an oscillatory integral. 
By Lemma \ref{L_star_l}, Theorem \ref{Lax02} in \S \ref{section_Existence of oscillatory integrals} and Lemma \ref{Generalized Fresnel integrals}, 
we obtain the following theorem: 
\begin{thm}[\cite{Nagano01},\cite{Nagano-Miyazaki03},\cite{Nagano-Miyazaki04}] 
\label{th01} 
Assume that $p,q \in \mathbb{C}$. 
\begin{enumerate} 
\item[(i)] 
If $p > 0$ and $q > 0$, 
then 
\begin{align} 
\tilde{I}_{p,q}^{\pm} 
:= Os\mbox{-}\int_{0}^{\infty} e^{\pm ix^{p}} x^{q-1} dx 
= p^{-1} e^{\pm i\frac{\pi}{2} \frac{q}{p}} \varGamma \left( \frac{q}{p} \right). 
\label{generalized_Fresnel_integral_def} 
\end{align} 
\item[(ii)] 
The $\tilde{I}_{p,q}^{\pm}$ can be extended non-zero meromorphic on $\mathbb{C}$ with poles of order 1 at $q = -pj$ for $j \in \mathbb{N}$ as to $q$ for each $p \in \mathbb{C} \setminus \{ 0 \}$, 
and meromorphic on $\mathbb{C} \setminus \{ 0 \}$ with poles of order 1 at $p = -q/j$ for $j \in \mathbb{N}$ as to $p$ for each $q \in \mathbb{C}$. 
\end{enumerate} 

We call $\tilde{I}^{\pm}_{p,q}$ ``generalized Fresnel integrals.'' 
\end{thm} 

\begin{proof} 
Since the lower side of the double sign $\pm$'s can be obtained as the conjugate of the upper one, we shall show the upper one. 
Since $a \equiv 1 \in \mathcal{A}^{\tau}_{\delta}(\mathbb{R})$ with $\tau \geq 0$, 
we can use Theorem \ref{Lax02} in \S \ref{section_Existence of oscillatory integrals} when $\lambda = 1$, $k=0$. 

(i) 
Suppose $p > 0$ and $q > 0$. 
By Theorem \ref{Lax02} (iv) in \S \ref{section_Existence of oscillatory integrals}, there exists the oscillatory integral: 
\begin{align} 
\tilde{I}_{p,q}^{+} 
:= \tilde{I}_{p,q}^{+}[1](1) 
:= Os\text{-}\int_{0}^{\infty} e^{ix^{p}} x^{q-1} dx 
:= \lim_{\varepsilon \to +0} \int_{0}^{\infty} e^{ix^{p}} x^{q-1} \chi_{\varepsilon}(x) dx, 
\notag 
\end{align} 
where $\chi \in \mathcal{S}(\mathbb{R})$ with $\chi(0) = 1$ and $\chi_{\varepsilon}(x) := \chi(\varepsilon x)$ for $x \in \mathbb{R}$ and $0 < \varepsilon <1 $. 

When $p > q$, put $f(x) = e^{ix^{p}} x^{q-1} \chi_{\varepsilon}(x)$. 
Since $f \in L^{1}(0,1]$ independent of $0 < \varepsilon <1 $, 
by using Lebesgue's dominated convergence theorem with Proposition \ref{chi epsilon} (i) in \S \ref{Preliminary}, we have 
\begin{align} 
Os\text{-}\int_{0}^{1} e^{ix^{p}} x^{q-1} dx 
= \int_{0}^{1} e^{ix^{p}} x^{q-1} dx. 
\label{Os_0_1} 
\end{align} 
Since $f(x) = O(x^{\beta})~(x \to \infty)$ with $\beta = q-1-m$ for any $m \in \mathbb{Z}_{\geq 0}$, 
taking $m$ such that $\beta < -1$, then $\int_{1}^{\infty} f(x) dx$ is absolutely convergent. Thus by integration by parts, we have 
\begin{align} 
&\int_{1}^{\infty} e^{ix^{p}} x^{q-1} \chi_{\varepsilon}(x) dx 
= \lim_{v \to \infty} \int_{1}^{v} \frac{1}{ip} \frac{d}{dx} (e^{ix^{p}}) x^{q-p} \chi_{\varepsilon}(x) dx \notag \\ 
&= \lim_{v \to \infty} \frac{1}{ip} \bigg\{ \Big[ e^{ix^{p}} x^{q-p} \chi_{\varepsilon}(x) \Big] _{1}^{v} \bigg. \notag \\ 
&\bigg. \hspace{0.5cm}- (q-p) \int_{1}^{v} e^{ix^{p}} x^{q-p-1} \chi_{\varepsilon}(x) dx - \int_{1}^{v} e^{ix^{p}} x^{q-p} \chi'_{\varepsilon}(x) dx \bigg\}. 
\notag 
\end{align} 
Put $g_{j}(x) = e^{ix^{p}} x^{q-p-1+j} \chi^{(j)}_{\varepsilon}(x)$ for $j=0,1$. 
Since $g_{j}(x) = O(x^{\beta'})~(x \to \infty)$ with $\beta' = q-p-1+j-m$ for any $m \in \mathbb{Z}_{\geq 0}$, 
taking $m$ such that $\beta' < -1$, then $\int_{1}^{\infty} g_{j}(x) dx$ is absolutely convergent, 
and $| e^{ix^{p}} x^{q-p} \chi_{\varepsilon}(x) | \to 0$ as $x \to \infty$. 
Hence 
\begin{align} 
&\int_{1}^{\infty} e^{ix^{p}} x^{q-1} \chi_{\varepsilon}(x) dx \notag \\ 
&= \frac{1}{ip} \bigg\{ -e^{i}\chi_{\varepsilon}(1) - (q-p) \int_{1}^{\infty} e^{ix^{p}} x^{q-p-1} \chi_{\varepsilon}(x) dx - \int_{1}^{\infty} e^{ix^{p}} x^{q-p} \chi'_{\varepsilon}(x) dx \bigg\}. 
\notag 
\end{align} 
By Proposition \ref{chi epsilon} (ii) in \S \ref{Preliminary}, there exists $C_{j}$ independent of $0 < \varepsilon <1 $ such that for any $x \in [1,\infty)$, 
\begin{align} 
| g_{j}(x) | 
\leq |x|^{q-p-1+j} C_{j} \langle x \rangle^{-j} 
\leq C_{j} |x|^{q-p-1} =: M_{j}(x). 
\notag 
\end{align} 
Since $\int_{1}^{\infty} M_{j}(x) dx$ is absolutely convergent independent of $\chi$, 
by using Lebesgue's dominated convergence theorem with Proposition \ref{chi epsilon} (i) and (iii) in \S \ref{Preliminary}, there exists the following oscillatory integral, and by \eqref{int_0_infty_e_ixp_x_q-1}, the following holds: 
\begin{align} 
&Os\text{-}\int_{1}^{\infty} e^{ix^{p}} x^{q-1} dx 
= \frac{1}{ip} \left\{ -e^{i} - (q-p) \int_{1}^{\infty} e^{ix^{p}} x^{q-p-1} dx \right\} 
= \int_{1}^{\infty} e^{ix^{p}} x^{q-1} dx. 
\label{Os_1_infty} 
\end{align} 
Therefore by \eqref{Os_0_1}, \eqref{Os_1_infty} and \eqref{I_pq}, 
\begin{align} 
\tilde{I}_{p,q}^{+} 
:= Os\text{-}\int_{0}^{\infty} e^{ix^{p}} x^{q-1} dx 
= \int_{0}^{\infty} e^{ix^{p}} x^{q-1} dx 
=: I_{p,q}^{+} 
= p^{-1} e^{i\frac{\pi}{2} \frac{q}{p}} \varGamma \left( \frac{q}{p} \right). 
\label{Generalized Fresnel Oscillatory integral p > q > 0} 
\end{align} 

When $q=p$, by integration by parts, 
\begin{align} 
\tilde{I}_{p,p}^{+} 
&= \lim_{\varepsilon \to +0} \lim_{\substack{u \to +0\\ v \to \infty}} \int_{u}^{v} e^{ix^{p}} x^{p-1} \chi_{\varepsilon}(x) dx 
= \lim_{\varepsilon \to +0} \lim_{\substack{u \to +0\\ v \to \infty}} \int_{u}^{v} \frac{1}{ip} \frac{d}{dx} (e^{ix^{p}}) \chi_{\varepsilon}(x) dx \notag \\ 
&= \lim_{\varepsilon \to +0} \lim_{\substack{u \to +0\\ v \to \infty}} \frac{1}{ip} \left( \Big[ e^{ix^{p}} \chi_{\varepsilon}(x) \Big]_{u}^{v} 
- \int_{u}^{v} e^{ix^{p}} \chi'_{\varepsilon}(x) dx \right). 
\notag 
\end{align} 
Here since $\chi_{\varepsilon} \in \mathcal{S}(\mathbb{R})$ for $0 < \varepsilon <1$, $\int_{0}^{\infty} e^{ix^{p}} \chi'_{\varepsilon}(x) dx$ is absolutely convergent and $|e^{ix^{p}} \chi_{\varepsilon}(x)| \to 0$ as $x \to \infty$. 
Hence by Theorem \ref{Lax02} (iii) in \S \ref{section_Existence of oscillatory integrals}, 
\begin{align} 
\tilde{I}_{p,p}^{+} 
&= \lim_{\varepsilon \to +0} \frac{1}{ip} \left( -1 - \int_{0}^{\infty} e^{ix^{p}} \chi'_{\varepsilon}(x) dx \right) 
= \frac{i}{p} 
= p^{-1} e^{i\frac{\pi}{2} \frac{p}{p}} \varGamma \left( \frac{p}{p} \right). 
\label{I_pp} 
\end{align} 

When $q>p$, 
let $l_{0} = [q/p)$. Then $0 <q-pl_{0} \leq p$. 
By Lemma \ref{L_star_l} (iv) and (i), Theorem \ref{Lax02} (iii) and \eqref{C_l_0_02} in \S \ref{section_Existence of oscillatory integrals}, 
\begin{align} 
\tilde{I}_{p,q}^{+} 
&= \lim_{\varepsilon \to +0} \int_{0}^{\infty} e^{ix^{p}} x^{q-1} \chi_{\varepsilon}(x) dx 
= \lim_{\varepsilon \to +0} \int_{0}^{\infty} e^{ix^{p}} L^{*l_{0}}(x^{q-1} \chi_{\varepsilon}(x)) dx \notag \\ 
&= \lim_{\varepsilon \to +0} \int_{0}^{\infty} e^{ix^{p}} \left( \frac{i}{p} \right)^{l_{0}} \sum_{j=0}^{l_{0}} C_{l_{0},j} x^{q-1-pl_{0}+j} \chi_{\varepsilon}^{(j)}(x) dx \notag \\ 
&= \left( \frac{i}{p} \right)^{l_{0}} C_{l_{0},0} \lim_{\varepsilon \to +0} \int_{0}^{\infty} e^{ix^{p}} x^{q-pl_{0}-1} \chi_{\varepsilon}(x) dx \notag \\ 
&= \left( \frac{i}{p} \right)^{l_{0}} \prod_{s=1}^{l_{0}} (q-ps) \tilde{I}_{p,q-pl_{0}}^{+} 
= e^{i\frac{\pi}{2} l_{0}} \prod_{s=1}^{l_{0}} \left( \frac{q}{p}-s \right) \tilde{I}_{p,q-pl_{0}}^{+}. 
\label{I_pq02} 
\end{align} 
When $q-pl_{0} = p$, that is, $q=p(l_{0}+1)$, then by \eqref{I_pq02} with \eqref{I_pp}, 
\begin{align} 
\tilde{I}_{p,q}^{+} 
&= \tilde{I}_{p,p(l_{0}+1)}^{+} 
= e^{i\frac{\pi}{2} l_{0}} \prod_{s=1}^{l_{0}} \{ (l_{0}+1)-s \} \tilde{I}_{p,p}^{+} 
= e^{i\frac{\pi}{2} l_{0}} l_{0}! \frac{i}{p} \notag \\ 
&= p^{-1} e^{i\frac{\pi}{2} (l_{0}+1)} \varGamma (l_{0}+1) 
= p^{-1} e^{i\frac{\pi}{2} \frac{q}{p}} \varGamma \left( \frac{q}{p} \right). 
\notag 
\end{align} 
When $q-pl_{0} < p$, then 
by \eqref{I_pq02} with \eqref{Generalized Fresnel Oscillatory integral p > q > 0}, 
\begin{align} 
\tilde{I}_{p,q}^{+} 
&= e^{i\frac{\pi}{2} l_{0}} \prod_{s=1}^{l_{0}} \left( \frac{q}{p}-s \right) p^{-1} e^{i\frac{\pi}{2} \frac{q-pl_{0}}{p}} \varGamma \left( \frac{q-pl_{0}}{p} \right) \notag \\ 
&= p^{-1} e^{i\frac{\pi}{2} l_{0}} e^{i\frac{\pi}{2} \left( \frac{q}{p} - l_{0} \right)} \prod_{s=1}^{l_{0}} \left( \frac{q}{p}-s \right) \varGamma \left( \frac{q}{p} - l_{0} \right) 
= p^{-1} e^{i\frac{\pi}{2} \frac{q}{p}} \varGamma \left( \frac{q}{p} \right). 
\notag 
\end{align} 

(ii) 
The function $e^{\pm i\frac{\pi}{2} z}$ are non-zero holomorphic on $\mathbb{C}$. 
The Gamma function $\varGamma (z)$ can be extended non-zero meromorphic on $\mathbb{C}$ with poles of order 1 at $z = -j$ for $j \in \mathbb{N}$ by analytic continuation. 
The function $f(q) = q/p$ is holomorphic on $\mathbb{C}$ for each $p \in \mathbb{C} \setminus \{ 0 \}$. 
The function $g(p) = q/p$ is holomorphic on $\mathbb{C} \setminus \{ 0 \}$ for each $q \in \mathbb{C}$. 
Therefore 
$\tilde{I}_{p,q}^{+} = p^{-1} e^{i\frac{\pi}{2} f(q)} \varGamma (f(q))$ can be extended non-zero meromorphic on $\mathbb{C}$ with poles of order 1 at $q = -pj$ for $j \in \mathbb{N}$ as to $q$ for each $p \in \mathbb{C} \setminus \{ 0 \}$, 
and $\tilde{I}_{p,q}^{+} = p^{-1} e^{i\frac{\pi}{2} g(p)} \varGamma (g(p))$ can be extended meromorphic on $\mathbb{C} \setminus \{ 0 \}$ with poles of order 1 at $p = -q/j$ for $j \in \mathbb{N}$ as to $p$ for each $q \in \mathbb{C}$. 
\end{proof} 

Using the theorem above, 
we can extend the Euler Beta function as follows. 
\begin{prop}[\cite{Nagano-Miyazaki03},\cite{Nagano-Miyazaki04}] 
Assume that $p_{j} > 0$ and $q_{j} 
\in \mathbb{C} \setminus \{ -p_{j}\mathbb{N} \}$ for $j=1,2,3$. 
Let 
\begin{align} 
\tilde{B}^{\pm}(p_{1},p_{2},p_{3};q_{1},q_{2},q_{3}) 
:= e^{\mp i \frac{\pi}{2} \left( \frac{q_{1}}{p_{1}} + 
\frac{q_{2}}{p_{2}} - \frac{q_{3}}{p_{3}} \right)} 
\frac{p_{1} p_{2}}{p_{3}} \frac{\tilde{I}_{p_{1},q_{1}}^{\pm} 
\tilde{I}_{p_{2},q_2}^{\pm}}{\tilde{I}_{p_{3},q_{3}}^{\pm}}. \notag 
\end{align} 
Then 
\begin{align} 
\tilde{B}^{\pm}(1,1,1;q_{1},q_{2},q_{1}+q_{2}) = B(q_{1},q_{2}), \notag 
\end{align} 
where $B(x,y)$ is the Euler Beta function, 
$\tilde{I}_{p_{1},q_{1}}^{\pm}, \tilde{I}_{p_{2},q_2}^{\pm}$ and $\tilde{I}_{p_{3},q_{3}}^{\pm}$ are generalized Fresnel integrals defined by \eqref{generalized_Fresnel_integral_def}. 
\end{prop}

\begin{proof} 
If $p_{j} = 1$ for $j=1,2,3$, 
since $q_{1}+q_{2} \in \mathbb{C} \setminus \{ -\mathbb{N} \}$, 
by Theorem \ref{th01} (ii), 
\begin{align} 
\tilde{B}^{\pm}(1,1,1;q_{1},q_{2},q_{1}+q_{2}) 
&= \frac{\tilde{I}_{1,q_{1}}^{\pm} \tilde{I}_{1,q_2}^{\pm}}{\tilde{I}_{1,q_{1}+q_{2}}^{\pm}} 
= \frac{e^{\mp i \frac{\pi}{2} q_{1}} \tilde{I}_{1,q_{1}}^{\pm} \cdot e^{\mp i \frac{\pi}{2} q_{2}} \tilde{I}_{1,q_{2}}^{\pm}}
{e^{\mp i \frac{\pi}{2} (q_{1}+q_{2}) } \tilde{I}_{1,q_{1}+q_{2}}^{\pm}} \notag \\ 
&= \frac{\varGamma (q_{1}) \varGamma (q_{2})}{\varGamma (q_{1}+q_{2})} 
= B(q_{1},q_{2}). \notag 
\end{align} 
\end{proof} 

\section{Applications to asymptotic expansions} 
\label{Applications_to_asymptotic expansions} 

In this section, 
we consider applications of generalized Fresnel integrals to asymptotic expansions of oscillatory integrals, which give extension of the method of stationary phase. 

By Theorem \ref{Lax02} and Corollary \ref{Os_m} in \S \ref{section_Existence of oscillatory integrals} with Theorem \ref{th01} in \S \ref{section_Generalized_Fresnel_Integrals}, 
we obtain the following theorem: 
\begin{thm}[\cite{Nagano-Miyazaki03},\cite{Nagano-Miyazaki04}] 
\label{th02} 
Assume that 
$\lambda > 0$ and $a \in \mathcal{A}^{\tau}_{\delta}(\mathbb{R})$. 
Let 
$\varphi \in C^{\infty}_{0}(\mathbb{R})$ be a cut-off function such that $\varphi \equiv 1$ on $|x| \leq r_{0}$ with $r_{0} \geq 1$ and $\varphi \equiv 0$ on $|x| \geq r_{1} > r_{0}$, and $\psi := 1 - \varphi$. 
Then the following hold: 
\begin{enumerate} 
\item[(i)] 
If $p>0$, then for any $N \in \mathbb{N}$, 
\begin{align} 
\tilde{I}_{p,1}^{\pm}[a](\lambda) 
:= Os\text{-}\int_{\bold{0}}^{\infty} e^{\pm i\lambda x^{p}} a(x) dx 
= \sum_{k=0}^{N-1} \tilde{I}_{p,k+1}^{\pm} \frac{a^{(k)}(0)}{k!} \lambda ^{-\frac{k+1}{p}} + R^{\pm}_{p,N} (\lambda) 
\label{j=0_integral} 
\end{align} 
with 
\begin{align} 
R^{\pm}_{p,N} (\lambda) 
= \sum_{k=N}^{N+[p]-1} \tilde{I}_{p,k+1}^{\pm} \frac{a^{(k)}(0)}{k!} \lambda ^{-\frac{k+1}{p}} 
+ \tilde{I}_{p,N+[p]+1}^{\pm}[\tilde{a}](\lambda) 
+ \tilde{I}_{p,1}^{\pm}[a \psi](\lambda) 
\label{j=0_remainder} 
\end{align} 
and 
\begin{align} 
\tilde{a}(x) 
:= \frac{1}{(N+[p]-1)!} \int_{0}^{1} (1-\theta)^{N+[p]-1} (a\varphi)^{(N+[p])} (\theta x) d\theta, 
\label{tilde_a_def} 
\end{align} 
where 
\begin{align} 
\tilde{I}_{p,k+1}^{\pm} 
:= Os\text{-}\int_{0}^{\infty} e^{\pm ix^{p}} x^{k} dx 
= p^{-1} e^{\pm i\frac{\pi}{2} \frac{k+1}{p}} \varGamma \left( \frac{k+1}{p} \right). 
\label{tilde_I_p_k+1_pm} 
\end{align} 
Then there exists a positive constant $M_{p,N}$ such that for any $\lambda \geq 1$, 
\begin{align} 
&|R^{\pm}_{p,N} (\lambda)| \leq M_{p,N} \lambda^{-\frac{N+1-(p-[p])}{p}}. 
\label{j=0_remainder02} 
\end{align} 

\item[(ii)] 
If $m \in \mathbb{N}$, 
then for any $N \in \mathbb{N}$, 
\begin{align} 
\tilde{J}_{m,1}^{\pm}[a](\lambda) 
:= Os\text{-}\int_{\bm{-\infty}}^{\infty} e^{\pm i\lambda x^{m}} a(x) dx 
= \sum_{k=0}^{N-1} c_{m,k}^{\pm} \frac{a^{(k)}(0)}{k!} \lambda ^{-\frac{k+1}{m}} + \tilde{R}^{\pm}_{m,N} (\lambda) 
\label{Os_int_m_01} 
\end{align} 
with 
\begin{align} 
&\tilde{R}^{\pm}_{m,N} (\lambda) 
= \sum_{k=N}^{N+m-1} c_{m,k}^{\pm} \frac{a^{(k)}(0)}{k!} \lambda ^{-\frac{k+1}{m}} 
+ \tilde{J}_{m,N+m+1}^{\pm}[\tilde{a}](\lambda) 
+ \tilde{J}_{m,1}^{\pm}[a \psi](\lambda) 
\label{j=0_remainder_m} 
\end{align} 
and 
\begin{align} 
c_{m,k}^{\pm} 
:&= \tilde{I}_{m,k+1}^{\pm} + (-1)^{k} \tilde{I}_{m,k+1}^{\pm \pm^{m}}, 
\label{c_mk^pm} 
\end{align} 
where 
\begin{align} 
\tilde{I}_{m,k+1}^{\pm \pm^{m}} 
:= Os\mbox{-}\int_{0}^{\infty} e^{\pm (-1)^{m} iy^{m}} y^{k} dy 
= m^{-1} e^{\pm (-1)^{m} i\frac{\pi}{2} \frac{k+1}{m}} \varGamma \left( \frac{k+1}{m} \right). 
\label{generalized_Fresnel_integral_def_pm_pm_m_type02} 
\end{align} 
And then there exists a positive constant $\tilde{M}_{m,N}$ such that for any $\lambda \geq 1$, 
\begin{align} 
&|\tilde{R}^{\pm}_{m,N} (\lambda)| \leq \tilde{M}_{m,N} \lambda^{-\frac{N+1}{m}}. 
\label{j=0_remainder_m_est} 
\end{align} 
\end{enumerate} 
\end{thm} 

\begin{proof} 
(i) 
Suppose $p > 0$. 
By Theorem \ref{Lax02} (iv) in \S \ref{section_Existence of oscillatory integrals}, 
there exist oscillatory integrals $\tilde{I}_{p,1}^{\pm}[a](\lambda)$, $\tilde{I}_{p,1}^{\pm}[a \varphi](\lambda)$, $\tilde{I}_{p,1}^{\pm}[a \psi](\lambda)$ with 
\begin{align} 
\tilde{I}_{p,1}^{\pm}[a](\lambda) 
= \tilde{I}_{p,1}^{\pm}[a \varphi](\lambda) + \tilde{I}_{p,1}^{\pm}[a \psi](\lambda). 
\label{I_1+I_2} 
\end{align} 

As to $\tilde{I}_{p,1}^{\pm}[a \varphi](\lambda)$ in \eqref{I_1+I_2}, applying Taylor's formula to $a(x) \varphi(x)$ at $x=0$, we have 
\begin{align} 
\tilde{I}_{p,1}^{\pm}[a \varphi](\lambda) 
&= Os\text{-}\int_{0}^{\infty} e^{\pm i\lambda x^{p}} \Bigg( \sum_{k=0}^{N+[p]-1} \frac{a^{(k)}(0)}{k!} x^{k} + x^{N+[p]} \tilde{a}(x) \Bigg) dx, 
\notag 
\end{align} 
where $\tilde{a}(x) $ is defined by \eqref{tilde_a_def}. 
Since $a\varphi \in C^{\infty}_{0}(\mathbb{R})$ in \eqref{tilde_a_def} with \eqref{inclusion_rel_01}, $\tilde{a} \in C^{\infty}_{0}(\mathbb{R}) \subset \mathcal{A}^{\tau}_{\delta}(\mathbb{R})$. 
Hence by Theorem \ref{Lax02} (iv) in \S \ref{section_Existence of oscillatory integrals}, there exist oscillatory integrals $\tilde{I}_{p,k+1}^{\pm}[1](\lambda)$ for $k=0,\dots,N+[p]-1$ and $\tilde{I}_{p,N+[p]+1}^{\pm}[\tilde{a}](\lambda)$ with 
\begin{align} 
\tilde{I}_{p,1}^{\pm}[a \varphi](\lambda) 
= \sum_{k=0}^{N+[p]-1} \frac{a^{(k)}(0)}{k!} Os\text{-}\int_{0}^{\infty} e^{\pm i\lambda x^{p}} x^{k} dx 
+ \tilde{I}_{p,N+[p]+1}^{\pm}[\tilde{a}](\lambda). 
\notag 
\end{align} 
Changing of variable $x=\lambda^{-1/p}y$ with Theorem \ref{th01} (i) in \S \ref{section_Generalized_Fresnel_Integrals} leads to 
\begin{align} 
\tilde{I}_{p,1}^{\pm}[a \varphi](\lambda) 
&= \sum_{k=0}^{N+[p]-1} \tilde{I}_{p,k+1}^{\pm} \frac{a^{(k)}(0)}{k!} \lambda ^{-\frac{k+1}{p}} 
+ \tilde{I}_{p,N+[p]+1}^{\pm}[\tilde{a}](\lambda), 
\label{I_1} 
\end{align} 
where $\tilde{I}_{p,k+1}^{\pm}$ is defined by \eqref{tilde_I_p_k+1_pm}. 

As to $\tilde{I}_{p,N+[p]+1}^{\pm}[\tilde{a}](\lambda)$ in \eqref{I_1}, 
since $N+[p]+1 > p$, by Theorem \ref{Lax02} (v) in \S \ref{section_Existence of oscillatory integrals}, for any $\lambda \geq 1$, 
\begin{align} 
| \tilde{I}_{p,N+[p]+1}^{\pm}[\tilde{a}](\lambda) | 
&\leq C_{p,N+[p]+1} |\tilde{a}|^{(\tau)}_{l_{0} + l_{p,N+[p]+1}} \lambda^{-\frac{N+1-(p-[p])}{p}}, 
\label{I_2_lambda_est} 
\end{align} 
where $C_{p,N+[p]+1}$ is a positive constant, $l_{0} := [(N+[p]+1)/p)$ and $l_{p,N+[p]+1} := [(N+[p]+1+\tau)^{+} /(p-1-\delta)]+1$. 

By \eqref{I_2_lambda_est}, any term in \eqref{I_1} where the order of $\lambda$ is less than or equal to $-\frac{N+1-(p-[p])}{p}$ can be turned into the remainder term. 
In fact, since $[p] \leq p < [p]+1$, for any $\lambda \geq 1$, 
\begin{align} 
\lambda^{-\frac{N+1}{p}} \leq \lambda^{-\frac{N+1-(p-[p])}{p}} < \lambda^{-\frac{N}{p}}. 
\label{I_1_lambda_est02} 
\end{align} 

As to $\tilde{I}_{p,1}^{\pm}[a \psi](\lambda)$ in \eqref{I_1+I_2}, 
by Theorem \ref{Lax02} (ii) in \S \ref{section_Existence of oscillatory integrals}, 
for each $l \in \mathbb{N}$ such that $l \geq l_{p,1}$, there exists a positive constant $\tilde{C}^{(0)}_{l}$ such that for any $\lambda > 0$, 
\begin{align} 
| \tilde{I}_{p,1}^{\pm}[a \psi](\lambda) | 
\leq \tilde{C}^{(0)}_{l} |a|^{(\tau)}_{l} \lambda^{-l}, 
\label{tilde_I_p_1_est_02} 
\end{align} 
where $l_{p,1} := [(1+\tau)^{+} /(p-1-\delta)]+1$. 

Therefore 
taking $l$ such that $l \geq \max \{ l_{p,1},\frac{N+1-(p-[p])}{p} \}$ and defining $R^{\pm}_{p,N} (\lambda)$ by \eqref{j=0_remainder}, according to \eqref{I_1+I_2}, \eqref{I_1}, \eqref{I_2_lambda_est}, \eqref{I_1_lambda_est02} and \eqref{tilde_I_p_1_est_02}, 
for any $N \in \mathbb{N}$, we obtain \eqref{j=0_integral} and \eqref{j=0_remainder02}. 

(ii) 
If $m \in \mathbb{N}$, by \eqref{Os_int_-infty_0_e^pm_i_lambda_x^m_a(x)_dx_def01} in \S \ref{section_Existence of oscillatory integrals} and (i), for any $N \in \mathbb{N}$, we have 
\begin{align} 
Os\text{-}\int_{-\infty}^{0} e^{\pm i\lambda x^{m}} a(x) dx 
&= Os\text{-}\int_{0}^{\infty} e^{\pm (-1)^{m}i\lambda y^{m}} a(-y) dy \notag \\ 
&= \sum_{k=0}^{N-1} (-1)^{k} \tilde{I}_{m,k+1}^{\pm \pm^{m}} \frac{a^{(k)}(0)}{k!} \lambda ^{-\frac{k+1}{m}} + R^{\pm \pm^{m}}_{m,N} (\lambda) 
\label{Os_int_-infty_0_e_i_lambda_x_m_a(x)_dx02} 
\end{align} 
with 
\begin{align} 
&R^{\pm \pm^{m}}_{m,N} (\lambda) 
= Os\text{-}\int_{-\infty}^{0} e^{\pm i\lambda x^{m}} a (x) \psi (x) dx 
+ \sum_{k=N}^{N+m-1} (-1)^{k} \tilde{I}_{m,k+1}^{\pm \pm^{m}} \frac{a^{(k)}(0)}{k!} \lambda ^{-\frac{k+1}{m}} \notag \\ 
&+ Os\text{-}\int_{-\infty}^{0} e^{\pm i\lambda x^{m}} \frac{x^{N+m}}{(N+m-1)!} \int_{0}^{1} (1-\theta)^{N+m-1} (a\varphi)^{(N+m)} (\theta x) d\theta dx, 
\label{j=0_remainder_m_2} 
\end{align} 
where $\tilde{I}_{m,k+1}^{\pm \pm^{m}}$ are defined by \eqref{generalized_Fresnel_integral_def_pm_pm_m_type02}. 
And then there exists a positive constant $M'_{m,N}$ such that for any $\lambda \geq 1$, 
\begin{align} 
&|\tilde{R}^{\pm \pm^{m}}_{m,N} (\lambda)| \leq M'_{m,N} \lambda^{-\frac{N+1}{m}}. 
\label{j=0_remainder_m_est02} 
\end{align} 
Therefore plugging \eqref{Os_int_-infty_0_e_i_lambda_x_m_a(x)_dx02}, \eqref{j=0_remainder_m_2} \eqref{j=0_remainder_m_est02} into \eqref{j=0_integral}, \eqref{j=0_remainder}, \eqref{j=0_remainder02} when $p=m$, by \eqref{Os_int_-infty_infty_e^pm_i_lambda_x^m_a(x)_dx_def01}, we obtain \eqref{Os_int_m_01}, \eqref{j=0_remainder_m} and \eqref{j=0_remainder_m_est}. 
\end{proof} 

The \eqref{Os_int_m_01} with \eqref{j=0_remainder_m} and \eqref{j=0_remainder_m_est} deduces the equations corresponding to (7.7.30) and (7.7.31) in \cite{Hormander02}. 
\begin{cor}[\cite{Nagano-Miyazaki04}] 
\label{cor01} 
Let $a \in \mathcal{A}^{\tau}_{\delta}(\mathbb{R})$. Then the following hold: 
\begin{enumerate} 
\item[(i)] 
If $m=2l-1$ for $l \in \mathbb{N}$, then for any $N \in \mathbb{N}$, 
\begin{align} 
&Os\text{-}\int_{-\infty}^{\infty} e^{\pm i\lambda x^{2l-1}} a(x) dx 
= \frac{2}{2l-1} \sum_{k=0}^{N-1} \left\{ \cos \frac{\pi (2k+1)}{2(2l-1)} \varGamma \left( \frac{2k+1}{2l-1} \right) \frac{a^{(2k)}(0)}{(2k)!} \lambda ^{-\frac{2k+1}{2l-1}} \right. \notag \\ 
&\hspace{2.95cm}\left. \pm i \sin \frac{\pi (2k+2)}{2(2l-1)} \varGamma \left( \frac{2k+2}{2l-1} \right) \frac{a^{(2k+1)}(0)}{(2k+1)!} \lambda ^{-\frac{2k+2}{2l-1}} \right\} + O\Big( \lambda ^{-\frac{N+1}{2l-1}} \Big) 
\label{Os_int_m=2l-1_01} 
\end{align} 
as $\lambda \to \infty$. 

\item[(ii)] 
If $m=2l$ for $l \in \mathbb{N}$, then for any $N \in \mathbb{N}$, 
\begin{align} 
Os\text{-}\int_{-\infty}^{\infty} e^{\pm i\lambda x^{2l}} a(x) dx 
= \frac{1}{l} \sum_{k=0}^{N-1} e^{\pm i\frac{\pi}{2} \frac{2k+1}{2l}} \varGamma \left( \frac{2k+1}{2l} \right) \frac{a^{(2k)}(0)}{(2k)!} \lambda ^{-\frac{2k+1}{2l}} + O\Big( \lambda ^{-\frac{N}{l}} \Big) 
\label{Os_int_m=2l_01} 
\end{align} 
as $\lambda \to \infty$. 
\end{enumerate} 
\end{cor} 

\begin{proof} 
(i) By \eqref{Os_int_m_01}, \eqref{j=0_remainder_m} and \eqref{c_mk^pm}, for any $N \in \mathbb{N}$, we have 
\begin{align} 
Os\text{-}\int_{-\infty}^{\infty} e^{\pm i\lambda x^{2l-1}} a(x) dx 
= \sum_{j=0}^{N-1} c_{2l-1,j}^{\pm} \frac{a^{(j)}(0)}{j!} \lambda ^{-\frac{j+1}{2l-1}} + O\Big( \lambda ^{-\frac{N+1}{2l-1}} \Big) 
\notag 
\end{align} 
as $\lambda \to \infty$, 
with 
\begin{align} 
c_{2l-1,j}^{\pm} 
&= (2l-1)^{-1} \left\{ e^{\pm i\frac{\pi}{2} \frac{j+1}{2l-1}} + (-1)^{j} e^{\mp i\frac{\pi}{2} \frac{j+1}{2l-1}} \right\} \varGamma \left( \frac{j+1}{2l-1} \right) \notag \\ 
&= 
\begin{cases} 
\dfrac{2}{2l-1} \cos \dfrac{\pi(2k+1)}{2(2l-1)} \varGamma \left( \dfrac{2k+1}{2l-1} \right), & \text{if $j=2k$}, \\ 
\dfrac{\pm 2i}{2l-1} \sin \dfrac{\pi(2k+2)}{2(2l-1)} \varGamma \left( \dfrac{2k+2}{2l-1} \right), & \text{if $j=2k+1$}, 
\end{cases} 
\notag 
\end{align} 
for $k \in \mathbb{Z}_{\geq 0}$. 
Here we used $e^{\pm i\theta} + e^{\mp i\theta} = 2\cos \theta$ and $e^{\pm i\theta} - e^{\mp i\theta} = \pm 2i\sin \theta$ for $\theta \in \mathbb{R}$. 
Hence 
we obtain \eqref{Os_int_m=2l-1_01}. 

(ii) By \eqref{Os_int_m_01}, \eqref{j=0_remainder_m} and \eqref{c_mk^pm}, for any $N \in \mathbb{N}$, we have 
\begin{align} 
Os\text{-}\int_{-\infty}^{\infty} e^{\pm i\lambda x^{2l}} a(x) dx 
= \sum_{j=0}^{(2N-1)-1} c_{2l,j}^{\pm} \frac{a^{(j)}(0)}{j!} \lambda ^{-\frac{j+1}{2l}} + O\Big( \lambda ^{-\frac{(2N-1)+1}{2l}} \Big) 
\notag 
\end{align} 
as $\lambda \to \infty$, 
with 
\begin{align} 
c_{2l,j}^{\pm} 
:= \tilde{I}_{2l,j+1}^{\pm} + (-1)^{j} \tilde{I}_{2l,j+1}^{\pm} 
&= 
\begin{cases} 
2 \tilde{I}_{2l,2k+1}^{\pm}, & \text{if $j=2k$}, \\ 
0, & \text{if $j=2k+1$}, 
\end{cases} 
\notag 
\end{align} 
for $k \in \mathbb{Z}_{\geq 0}$. 
Hence we obtain \eqref{Os_int_m=2l_01}. 
\end{proof} 
When $l=1$, Corollary \ref{cor01} (i) and (ii) are equivalent to the property that the Fourier transform of $a(x)$ is of $O(\lambda ^{-N})~(\lambda \to \infty)$ for any $N \in \mathbb{N}$ in $\mathcal{S}(\mathbb{R})$ and the method of stationary phase in one variable respectively. 
Therefore Theorem \ref{th02} (ii) can be considered to give the unified formula including the Fourier transform and the method of stationary phase in one variable. 

By using Theorem \ref{th02}, 
we obtain the following results: 
\begin{thm}[\cite{Nagano01},\cite{Nagano02}] 
\label{analytic_phase} 
Assume that $p > 0$. 
Let $\{ a_{j} \}_{j \in \mathbb{N}}$ be a sequence of real numbers 
such that $l_{0} := \sup_{j \in \mathbb{N}} |a_{j}|^{1/j} < \infty$. 
Then 
\begin{align} 
f(x) := x \bigg( 1 + \sum_{j=1}^{\infty} a_{j} x^{j} \bigg) ^{1/p} 
\label{f(x)=x(1+sum_j=1^infty_a_j_x^j)^1/p} 
\end{align} 
can be defined for $|x| < R_{0} := (2l_{0})^{-1}$ where $R_{0}=\infty$ if $l_{0}=0$, and 
there exists a connected open neighborhood $U$ of the origin in $\mathbb{R}$ such that $U \subset (-R_{0},R_{0})$ 
and a diffeomorphism $x = \varPhi (y)$ of class $C^{\infty}$ for $x,y \in U$ such that $f(x) = y$, $\varPhi (U \cap [0,\infty)) = U \cap [0,\infty)$ and $\varPhi (U \cap (-\infty,0]) = U \cap (-\infty,0]$, and for any $a \in C^{\infty}_{0}(\mathbb{R})$ such that $\mathrm{supp}~a \subset U$, 
the following hold: 

\begin{enumerate} 
\item[(i)] 
For any $N \in \mathbb{N}$, 
\begin{align} 
&\int_{0}^{\infty} e^{\pm i\lambda x^{p} \left( 1 + \sum_{j=1}^{\infty} a_{j} x^{j} \right)} a(x) dx \notag \\ 
&= \sum_{k=0}^{N-1} 
\frac{\tilde{I}_{p,k+1}^{\pm}}{k!} \left( \frac{d}{dy} \right)^{k} \bigg|_{y=0} \left\{ a(\varPhi (y)) \frac{d\varPhi}{dy}(y) \right\} \lambda ^{-\frac{k+1}{p}} 
+ O\Big( \lambda ^{-\frac{N+1-(p-[p])}{p}} \Big) 
\label{Os_int_analytic_p_01} 
\end{align} 
as $\lambda \to \infty$, 
where $\tilde{I}_{p,k+1}^{\pm}$ are defined by \eqref{tilde_I_p_k+1_pm}. 

\item[(ii)] 
If $m \in \mathbb{N}$, 
then for any $N \in \mathbb{N}$, 
\begin{align} 
&\int_{-\infty}^{\infty} e^{\pm i\lambda x^{m} \left( 1 + \sum_{j=1}^{\infty} a_{j} x^{j} \right)} a(x) dx \notag \\ 
&= \sum_{k=0}^{N-1} 
\frac{c_{m,k}^{\pm -}}{k!} \left( \frac{d}{dy} \right)^{k} \bigg|_{y=0} \left\{ a(\varPhi (y)) \frac{d\varPhi}{dy}(y) \right\} \lambda ^{-\frac{k+1}{m}} + O\Big( \lambda ^{-\frac{N+1}{m}} \Big) 
\label{Os_int_analytic_m_01} 
\end{align} 
as $\lambda \to \infty$, 
with 
\begin{align} 
c_{m,k}^{\pm -} 
:&= \tilde{I}_{m,k+1}^{\pm} - (-1)^{k} \tilde{I}_{m,k+1}^{\pm \pm^{m}}, \label{c_mk^pm02} 
\end{align} 
where $\tilde{I}_{m,k+1}^{\pm \pm^{m}}$ are defined by \eqref{generalized_Fresnel_integral_def_pm_pm_m_type02}. 
\end{enumerate} 
\end{thm} 

\begin{proof}
(i) 
Since $l := \mathrm{lim~sup}_{j\to \infty} |a_{j}|^{1/j} \leq \sup_{j \in \mathbb{N}} |a_{j}|^{1/j} =: l_{0} < \infty$, 
by Cauchy-Hadamard's formula, 
the radius of convergence of the power series $\sum_{j=1}^{\infty} a_{j} x^{j}$ is $R := 1/l >0$ where $R=\infty$ if $l=0$. 
Moreover 
if $|x| < R_{0} := (2l_{0})^{-1} \leq (2l)^{-1} < R$, 
then 
\begin{align} 
\bigg| \sum_{j=1}^{\infty} a_{j} x^{j} \bigg| 
\leq \sum_{j=1}^{\infty} |a_{j}| |x|^{j} 
< \sum_{j=1}^{\infty} |a_{j}| \Big\{ \big( 2 \sup_{j \in \mathbb{N}} |a_{j}|^{1/j} \big)^{-1} \Big\}^{j} 
\leq \sum_{j=1}^{\infty} 2^{-j} 
= 1. 
\notag 
\end{align} 
Hence 
if $|x| < R_{0}$, 
since $1 + \sum_{j=1}^{\infty} a_{j} x^{j} > 0$, 
then 
\eqref{f(x)=x(1+sum_j=1^infty_a_j_x^j)^1/p} can be defined for $x \in (-R_{0},R_{0})$. 
Since $f$ is a function of class $C^{\infty}$ on $(-R_{0},R_{0})$ with $f(0) = 0$ and $f'(0) = 1$, there exists a connected open neighborhood $U$ of the origin in $\mathbb{R}$ such that $U \subset (-R_{0},R_{0})$ and a diffeomorphism $x = \varPhi (y)$ of class $C^{\infty}$ for $x,y \in U$ such that $f(x) = y$, $\varPhi (U \cap [0,\infty)) = U \cap [0,\infty)$ and $\varPhi (U \cap (-\infty,0]) = U \cap (-\infty,0]$, and for any $a \in C^{\infty}_{0}(\mathbb{R})$ such that $\mathrm{supp}~a \subset U$, 
the following improper integral is absolutely convergent, 
and by change of variable $x = \varPhi (y)$ on $U$, 
since $x^{p} (1 + \sum_{j=1}^{\infty} a_{j} x^{j}) = (f(x))^{p} = y^{p}$, then the following holds: 
\begin{align} 
\int_{0}^{\infty} e^{\pm i\lambda x^{p} \left( 1 + \sum_{j=1}^{\infty} a_{j} x^{j} \right)} a(x) dx 
= \int_{0}^{\infty} e^{\pm i\lambda y^{p}} a(\varPhi (y)) \frac{d\varPhi}{dy}(y) dy. 
\notag 
\end{align} 
Here since $a(\varPhi (y)) \frac{d\varPhi}{dy}(y)$ can be extend to a function in $C^{\infty}_{0}(\mathbb{R})$, 
by Theorem \ref{Lax02} in \S \ref{section_Existence of oscillatory integrals} and Theorem \ref{th02} (i), 
for any $N \in \mathbb{N}$, we obtain \eqref{Os_int_analytic_p_01} as $\lambda \to \infty$. 

(ii) 
If $m \in \mathbb{N}$, by (i), the following improper integral is absolutely convergent, 
and by change of variable $x = \varPhi (y)$ on $U$ and $y=-u$, 
since $x^{m} (1 + \sum_{j=1}^{\infty} a_{j} x^{j}) = (f(x))^{m} = y^{m} = (-u)^{m}$, then the following holds: 
\begin{align} 
\int_{-\infty}^{0} e^{\pm i\lambda x^{m} \left( 1 + \sum_{j=1}^{\infty} a_{j} x^{j} \right)} a(x) dx 
&= -\int_{0}^{\infty} e^{\pm i\lambda (-u)^{m}} a(\varPhi(-u)) \frac{d\varPhi(-u)}{du} du. 
\notag 
\end{align} 
Here 
since $a(\varPhi (-u)) \frac{d\varPhi(-u)}{du}$ can be extend to a function in $C^{\infty}_{0}(\mathbb{R})$, 
by Theorem \ref{Lax02} in \S \ref{section_Existence of oscillatory integrals} and \eqref{Os_int_-infty_0_e_i_lambda_x_m_a(x)_dx02} 
with \eqref{j=0_remainder_m_est02}, 
for any $N \in \mathbb{N}$, 
\begin{align} 
&\int_{-\infty}^{0} e^{\pm i\lambda x^{m} \left( 1 + \sum_{j=1}^{\infty} a_{j} x^{j} \right)} a(x) dx 
\notag \\ 
&= -\sum_{k=0}^{N-1} (-1)^{k} \frac{\tilde{I}_{m,k+1}^{\pm \pm^{m}}}{k!} \left( \frac{d}{du} \right)^{k} \bigg|_{u=0} \left\{ a(\varPhi(-u)) \frac{d\varPhi(-u)}{du} \right\} \lambda ^{-\frac{k+1}{m}} 
+ O\Big( \lambda ^{-\frac{N+1}{m}} \Big) 
\label{int_-infty_0} 
\end{align} 
as $\lambda \to \infty$. 
Therefore plugging \eqref{int_-infty_0} into \eqref{Os_int_analytic_p_01} when $p=m$, 
for any $N \in \mathbb{N}$, 
we obtain \eqref{Os_int_analytic_m_01} with \eqref{c_mk^pm02} as $\lambda \to \infty$. 
\end{proof} 

In particular, when $a_{j}=1/j!$, by using differential coefficients of all orders of 
Lambert W function $X=W(Y)$ at $Y=0$, 
which is the multivalued inverse function of $Y=Xe^{X}$ with two blanches $X=W_{0}(Y)$ for $Y \geq -1/e$ and $X \geq -1$, and $X=W_{-1}(Y)$ for $-1/e \leq Y < 0$ and $X \leq -1$ (\cite{CGHJK}), 
then we can compute a term in the asymptotic expansion concretely as follows: 
\begin{cor}[\cite{Nagano02}]  
\label{Lambert_W_function} 
Assume that $p > 0$. 
Let $X=W_{0}(Y)$ be the blanch of Lambert W function for $Y \geq -1/e$ and $X \geq -1$, and $a \in C^{\infty}_{0}(\mathbb{R})$. Then the following hold:
\begin{enumerate} 
\item[(i)] 
For any $N \in \mathbb{N}$, 
\begin{align} 
&\int_{0}^{\infty} e^{\pm i\lambda x^{p} e^{x}} a(x) dx \notag \\ 
&= \sum_{k=0}^{N-1} 
\frac{\tilde{I}_{p,k+1}^{\pm}}{k!} \left( \frac{d}{dy} \right)^{k} \bigg|_{y=0} \left\{ a\bigg( p W_{0} \bigg( \frac{y}{p} \bigg) \bigg) \frac{dW_{0}}{dY} \bigg( \frac{y}{p} \bigg) \right\} \lambda ^{-\frac{k+1}{p}} 
+ O\Big( \lambda ^{-\frac{N+1-(p-[p])}{p}} \Big) 
\notag 
\end{align} 
as $\lambda \to \infty$, 
where $\tilde{I}_{p,k+1}^{\pm}$ are defined by \eqref{tilde_I_p_k+1_pm}. 

\item[(ii)] 
If $m \in \mathbb{N}$, 
then for any $N \in \mathbb{N}$, 
\begin{align} 
&\int_{-\infty}^{\infty} e^{\pm i\lambda x^{m} e^{x}} a(x) dx \notag \\ 
&= \sum_{k=0}^{N-1} 
\frac{c_{m,k}^{\pm -}}{k!} \left( \frac{d}{dy} \right)^{k} \bigg|_{y=0} \left\{ a\bigg( m W_{0} \bigg( \frac{y}{m} \bigg) \bigg) \frac{dW_{0}}{dY} \bigg( \frac{y}{m} \bigg) \right\} \lambda ^{-\frac{k+1}{m}} + O\Big( \lambda ^{-\frac{N+1}{m}} \Big) 
\notag 
\end{align} 
as $\lambda \to \infty$, 
where $c_{m,k}^{\pm -}$ are defined by \eqref{c_mk^pm02}. 
\end{enumerate} 
\end{cor} 

\begin{proof} 
If $a_{j}=1/j!$, since $\sup_{j \in \mathbb{N}} |1/j!|^{1/j} = \sup_{j \in \mathbb{N}} \prod_{k=1}^{j} (1/k)^{1/j} \leq 1$, 
then (i), (ii) of Theorem \ref{analytic_phase} hold. 
Let $y = x e^{\frac{x}{p}}$ for $x \in \mathbb{R}$, 
since $\frac{y}{p} = \big( \frac{x}{p} \big) e^{\frac{x}{p}}$, 
then $\frac{x}{p} = W_{0} \big( \frac{y}{p} \big)$ for $y \geq -p/e$ and $x \geq -p$. 
Hence $x = \varPhi (y) = pW_{0} \big( \frac{y}{p} \big)$ and $\frac{d\varPhi}{dy}(y) = \frac{dW_{0}}{dY} \big( \frac{y}{p} \big)$. 
Therefore (i) and (ii) hold. 
\end{proof} 

Finally by using Theorem \ref{th02} for $a \in C^{\infty}_{0}(\mathbb{R}^{n})$, 
we obtain asymptotic expansions of oscillatory integrals with degenerate phases including the types $A_{k}$, $E_6$, $E_8$ in multivariable. 
\begin{thm}[\cite{Nagano-Miyazaki02}] 
\label{multivariable} 
Let $a \in C^{\infty}_{0}(\mathbb{R}^{n})$. 
Then the following hold: 
\begin{enumerate} 
\item[(i)] 
If $p:=(p_{1},\dots,p_{n}) \in (0,\infty)^{n}$, 
then for any $N \in \mathbb{N}$, 
\begin{align} 
&\int_{(0,\infty)^{n}} e^{i\lambda \sum_{j=1}^{n} \pm_{j} x_{j}^{p_{j}}} a(x) dx \notag \\ 
&= \sum_{\alpha \in \Omega_{p}^{N}} \prod_{j=1}^{n} \tilde{I}^{\pm_{j}}_{p_{j},\alpha_{j}+1} \frac{\partial_{x}^{\alpha} a(0)}{\alpha!} \lambda^{-\sum_{j=1}^{n} \frac{\alpha_{j}+1}{p_{j}}} + O \Big( \lambda^{-\frac{N+1-\max (p_{j}-[p_{j}])}{\max p_{j}}} \Big) 
\label{tilde_I_phi_a_lambda} 
\end{align} 
as $\lambda \to \infty$, 
where $\pm_{j}$ represents ``$+$" or ``$-$" for each $j$, 
\begin{align} 
\Omega_{p}^{N} 
:= \bigg\{ \alpha = (\alpha_{1},\dots,\alpha_{n}) \in \mathbb{Z}_{\geq 0}^{n} \bigg| \sum_{j=1}^{n} \frac{\alpha_{j}+1}{p_{j}} < \frac{N+1-\max (p_{j}-[p_{j}])}{\max p_{j}} \bigg\} 
\label{Omega_k} 
\end{align} 
and 
\begin{align} 
\tilde{I}^{\pm_{j}}_{p_{j},\alpha_{j}+1} 
:= Os\mbox{-}\int_{0}^{\infty} e^{\pm_{j} ix_{j}^{p_{j}}} x_{j}^{\alpha_{j}} dx_{j} 
= p_{j}^{-1} e^{\pm_{j} i\frac{\pi}{2} \frac{\alpha_{j}+1}{p_{j}}} \varGamma \left( \frac{\alpha_{j}+1}{p_{j}} \right). 
\label{c_pm_j_j_alpha_j_0} 
\end{align} 

\item[(ii)] 
If $m:=(m_{1},\dots,m_{n}) \in \mathbb{N}^{n}$, 
then for any $N \in \mathbb{N}$, 
\begin{align} 
\int_{\mathbb{R}^{n}} e^{i\lambda \sum_{j=1}^{n} \pm_{j} x_{j}^{m_{j}}} a(x) dx 
&= \sum_{\alpha \in \Omega_{m}^{N}} \prod_{j=1}^{n} c^{\pm_{j}}_{m_{j},\alpha_{j}} \frac{\partial_{x}^{\alpha} a(0)}{\alpha!}  \lambda^{-\sum_{j=1}^{n} \frac{\alpha_{j}+1}{m_{j}}} + O \Big( \lambda^{-\frac{N+1}{\max m_{j}}} \Big) 
\notag 
\end{align} 
as $\lambda \to \infty$, 
with 
\begin{align} 
c^{\pm_{j}}_{m_{j},\alpha_{j}} 
:&= \tilde{I}_{m_{j},\alpha_{j}+1}^{\pm_{j}} + (-1)^{\alpha_{j}} \tilde{I}_{m_{j},\alpha_{j}+1}^{\pm_{j} \pm^{m_{j}}}, 
\notag 
\end{align} 
where 
\begin{align} 
\tilde{I}_{m_{j},\alpha_{j}+1}^{\pm_{j} \pm^{m_{j}}} 
:= Os\mbox{-}\int_{0}^{\infty} e^{\pm_{j} (-1)^{m_{j}} ix_{j}^{m_{j}}} x_{j}^{\alpha_{j}} dx_{j} 
= m_{j}^{-1} e^{\pm_{j} (-1)^{m_{j}} i\frac{\pi}{2} \frac{\alpha_{j}+1}{m_{j}}} \varGamma \left( \frac{\alpha_{j}+1}{m_{j}} \right). 
\notag 
\end{align} 
\end{enumerate} 
\end{thm} 

\begin{proof} 
(i) 
Since $a \in C^{\infty}_{0}(\mathbb{R}^{n})$, 
there exists a positive constant $r$ such that $\mathrm{supp}~a \subset \prod_{j=1}^{n} [-r,r]$. 
Let $\varphi_{j}(x_{j}) \in C^{\infty}_{0}(\mathbb{R})$ be a cut-off function such that $\varphi_{j} \equiv 1$ on $|x_{j}| \leq r_{0}$ with $r_{0} \geq \max \{ r,1 \}$ and $\varphi_{j} \equiv 0$ on $|x_{j}| \geq r_{1} > r_{0}$, and $\psi_{j}(x_{j}) := 1 - \varphi_{j}(x_{j})$ for $j=1,\dots,n$, 
and let 
\begin{align} 
&a_{1,\dots,j-1}(x_{j},\dots,x_{n}) 
:= \prod_{k=1}^{j-1} \frac{\partial_{x_{k}}^{\alpha_{k}}}{\alpha_{k}!} \bigg|_{x_{k}=0} \varphi_{k}(x_{k}) a(x) 
\label{a_1_j-1_x_j_def} 
\end{align} 
for $x = (x_{1},\dots,x_{n}) \in \mathbb{R}^{n}$ and $j=1,\dots,n+1$, where 
\begin{align} 
&a_{1,\dots,j-1}(x_{j},\dots,x_{n}) := a(x), & &\text{if $j=1$}, \notag \\ 
&a_{1,\dots,j-1}(x_{j},\dots,x_{n}) := \frac{\partial_{x}^{\alpha} a(0)}{\alpha!}, & &\text{if $j=n+1$}. 
\notag 
\end{align} 
Then since $a(x)\psi_{j}(x_{j}) \equiv 0$ for $x_{j} \in \mathbb{R}$, 
by Theorem \ref{th02} (i), for any $N \in \mathbb{N}$, 
\begin{align} 
&\int_{0}^{\infty} dx_{j} e^{\pm_{j} i\lambda x_{j}^{p_{j}}} a_{1,\dots,j-1}(x_{j},\dots,x_{n}) \notag \\ 
&= \sum_{\alpha_{j}=0}^{N-1} \tilde{I}_{p_{j},\alpha_{j}+1}^{\pm_{j}} \lambda^{-\frac{\alpha_{j}+1}{p_{j}}} a_{1,\dots,j}(x_{j+1},\dots,x_{n}) 
+ R^{\pm_{j}}_{p_{j},N} (\lambda;x_{j+1},\dots,x_{n}) 
\label{j=1_dots_n-1_integral} 
\end{align} 
with 
\begin{align} 
R^{\pm_{j}}_{p_{j},N} (\lambda;x_{j+1},\dots,x_{n}) 
&= \sum_{\alpha_{j}=N}^{N+[p_{j}]-1} \tilde{I}_{p_{j},\alpha_{j}+1}^{\pm_{j}} \lambda^{-\frac{\alpha_{j}+1}{p_{j}}} a_{1,\dots,j}(x_{j+1},\dots,x_{n}) \notag \\ 
&\hspace{0.90cm}+ \int_{0}^{\infty} dx_{j} e^{\pm_{j} i\lambda x_{j}^{p_{j}}} x_{j}^{N+[p_{j}]} \tilde{a}_{j}(x_{j},\dots,x_{n}) 
\label{j=1_dots_n-1_remainder} 
\end{align} 
and 
\begin{align} 
&\tilde{a}_{j}(x_{j},\dots,x_{n}) := \frac{1}{(N+[p_{j}]-1)!} \notag \\ 
&\hspace{0.0cm}\times \int_{0}^{1} d\theta_{j} (1-\theta_{j})^{N+[p_{j}]-1} \partial_{y_{j}}^{N+[p_{j}]} \bigg|_{y_{j} = \theta_{j} x_{j}} a_{1,\dots,j-1}(y_{j},x_{j+1},\dots,x_{n}) 
\label{tilde_a_j_x_j_def} 
\end{align} 
for $j=1,\dots,n$, 
where $\tilde{I}_{p_{j},\alpha_{j}+1}^{\pm_{j}}$ is defined by \eqref{c_pm_j_j_alpha_j_0}, and 
\begin{align} 
&R^{\pm_{j}}_{p_{j},N} (\lambda;x_{j+1},\dots,x_{n}) := R^{\pm_{n}}_{p_{n},N}(\lambda), & &\text{if $j=n$}, 
\notag \\ 
&a_{1,\dots,j-1}(y_{j},x_{j+1},\dots,x_{n}) := a_{1,\dots,n-1}(y_{n}), & &\text{if $j=n$}. 
\notag 
\end{align} 

Thanks to $a \in C^{\infty}_{0}(\mathbb{R}^{n})$, 
by \eqref{a_1_j-1_x_j_def}, \eqref{tilde_a_j_x_j_def} and \eqref{j=1_dots_n-1_remainder}, 
since $a_{1,\dots,j-1}, \tilde{a}_{j} \in C^{\infty}_{0}(\mathbb{R}^{n-j-1})$ and $R^{\pm_{j}}_{p_{j},N} \in C^{\infty}_{0}(\mathbb{R}^{n-j})$ as to $ (x_{j+1},\dots,x_{n})$, then 
the given integral can be reduced to an iterated integral, which can be decomposed into a product of asymptotic expansions of the integrals of each variable using \eqref{j=1_dots_n-1_integral} with \eqref{j=1_dots_n-1_remainder}. 
In fact, for any $N \in \mathbb{N}$, we have
\begin{align} 
I^{(0)}_{0} 
:&= \int_{(0,\infty)^{n}} e^{i\lambda \sum_{j=1}^{n} \pm_{j} x_{j}^{p_{j}}} a(x) dx 
= \prod_{j=1}^{n} \int_{0}^{\infty} dx_{j} e^{\pm_{j}i\lambda x_{j}^{p_{j}}} a(x) \notag \\ 
&= \prod_{j=2}^{n} \int_{0}^{\infty} dx_{j} e^{\pm_{j}i\lambda x_{j}^{p_{j}}} \notag \\ 
&\hspace{0.5cm}\times 
\Bigg( \sum_{\alpha_{1}=0}^{N-1} \tilde{I}_{p_{1},\alpha_{1}+1}^{\pm_{1}} \lambda^{-\frac{\alpha_{1}+1}{p_{1}}} a_{1}(x_{2},\dots,x_{n}) 
+ R^{\pm_{1}}_{p_{1},N} (\lambda;x_{2},\dots,x_{n}) \Bigg) \notag \\ 
&= I^{(1)}_{0} + I^{(1)}_{1} 
\label{I_0_0} 
\end{align} 
with 
\begin{align} 
I^{(1)}_{0} 
:&= \sum_{\alpha_{1}=0}^{N-1} \tilde{I}_{p_{1},\alpha_{1}+1}^{\pm_{1}} \lambda^{-\frac{\alpha_{1}+1}{p_{1}}} \prod_{j=2}^{n} \int_{0}^{\infty} dx_{j} e^{\pm_{j}i\lambda x_{j}^{p_{j}}} a_{1}(x_{2},\dots,x_{n}), \notag \\ 
I^{(1)}_{1} 
:&= \prod_{j=2}^{n} \int_{0}^{\infty} dx_{j} e^{\pm_{j}i\lambda x_{j}^{p_{j}}} R^{\pm_{1}}_{p_{1},N} (\lambda;x_{2},\dots,x_{n}). 
\notag 
\end{align} 
Similarly we have 
\begin{align} 
I^{(1)}_{0} 
&= \sum_{\alpha_{1}=0}^{N-1} \tilde{I}^{\pm_{1}}_{p_{1},\alpha_{1}+1} \lambda^{-\frac{\alpha_{1}+1}{p_{1}}} \prod_{j=3}^{n} \int_{0}^{\infty} dx_{j} e^{\pm_{j}i\lambda x_{j}^{p_{j}}} \notag \\ 
&\hspace{0.65cm}\times \Bigg( \sum_{\alpha_{2}=0}^{N-1} \tilde{I}^{\pm_{2}}_{p_{2},\alpha_{2}+1} \lambda^{-\frac{\alpha_{2}+1}{p_{2}}} a_{1,2}(x_{3},\dots,x_{n}) 
+ R^{\pm_{2}}_{p_{2},N} (\lambda;x_{3},\dots,x_{n}) \Bigg) \notag \\ 
&= I^{(2)}_{0} + I^{(2)}_{1} 
\notag 
\end{align} 
with 
\begin{align} 
I^{(2)}_{0} 
:&= \sum_{\alpha_{1}=0}^{N-1} \sum_{\alpha_{2}=0}^{N-1} \prod_{j=1}^{2} \tilde{I}^{\pm_{j}}_{p_{j},\alpha_{j}+1} \lambda^{-\sum_{j=1}^{2} \frac{\alpha_{j}+1}{p_{j}}} \prod_{j=3}^{n} \int_{0}^{\infty} dx_{j} e^{\pm_{j}i\lambda x_{j}^{p_{j}}} a_{1,2}(x_{3},\dots,x_{n}), \notag \\ 
I^{(2)}_{1} 
:&= \sum_{\alpha_{1}=0}^{N-1} \tilde{I}^{\pm_{1}}_{p_{1},\alpha_{1}+1} \lambda^{-\frac{\alpha_{1}+1}{p_{1}}} \prod_{j=3}^{n} \int_{0}^{\infty} dx_{j} e^{\pm_{j}i\lambda x_{j}^{p_{j}}} R^{\pm_{2}}_{p_{2},N} (\lambda;x_{3},\dots,x_{n}). 
\notag 
\end{align} 
Inductively we have 
\begin{align} 
I^{(0)}_{0} 
= I^{(n)}_{0} + \sum_{k=1}^{n} I^{(k)}_{1} 
\notag 
\end{align} 
with 
\begin{align} 
I^{(n)}_{0} 
:&= \sum_{\alpha_{1}=0}^{N-1} \cdots \sum_{\alpha_{n}=0}^{N-1} \prod_{j=1}^{n} \tilde{I}^{\pm_{j}}_{p_{j},\alpha_{j}+1} \frac{\partial_{x}^{\alpha} a(0)}{\alpha!} \lambda^{-\sum_{j=1}^{n} \frac{\alpha_{j}+1}{p_{j}}}, 
\label{I_n_0_def} \\ 
I^{(k)}_{1} 
:&= \sum_{\alpha_{1}=0}^{N-1} \cdots \sum_{\alpha_{k-1}=0}^{N-1} \prod_{j=1}^{k-1} \tilde{I}^{\pm_{j}}_{p_{j},\alpha_{j}+1} \lambda^{-\sum_{j=1}^{k-1} \frac{\alpha_{j}+1}{p_{j}}} F_{k}(\lambda), 
\label{I_k_1_def} 
\end{align} 
and 
\begin{align} 
F_{k}(\lambda) 
:= \prod_{j=k+1}^{n} \int_{0}^{\infty} dx_{j} e^{\pm_{j}i\lambda x_{j}^{p_{j}}} R^{\pm_{k}}_{p_{k},N} (\lambda;x_{k+1},\dots,x_{n}) 
\label{J_k_lambda_def} 
\end{align} 
for $k=1,\dots,n$, 
where 
\begin{align} 
&I^{(k)}_{1} := F_{1}(\lambda), &\text{if $k=1$}, \notag \\ 
&F_{k}(\lambda) := R^{\pm_{n}}_{p_{n},N} (\lambda), &\text{if $k=n$}. 
\notag 
\end{align} 

In order to estimate $I^{(k)}_{1}$, we shall estimate $F_{k}(\lambda)$. \\ 
Since $a_{1,\dots,k-1}, \tilde{a}_{k} \in C^{\infty}_{0}(\mathbb{R}^{n-k-1})$, from \eqref{J_k_lambda_def} with \eqref{j=1_dots_n-1_remainder}, we have 
\begin{align} 
F_{k}(\lambda) = F_{k1}(\lambda) + F_{k2}(\lambda) 
\label{J_k_lambda_decomposition} 
\end{align} 
with 
\begin{align} 
F_{k1}(\lambda) 
:&= \sum_{\alpha_{k}=N}^{N+[p_{k}]-1} \tilde{I}_{p_{k},\alpha_{k}+1}^{\pm_{k}} \lambda^{-\frac{\alpha_{k}+1}{p_{k}}} \prod_{j=k+1}^{n} \int_{0}^{\infty} dx_{j} e^{\pm_{j}i\lambda x_{j}^{p_{j}}} 
a_{1,\dots,k}(x_{k+1},\dots,x_{n}), \notag \\ 
F_{k2}(\lambda) 
:&= \prod_{j=k+1}^{n} \int_{0}^{\infty} dx_{j} e^{\pm_{j}i\lambda x_{j}^{p_{j}}} \int_{0}^{\infty} dx_{k} e^{\pm_{k} i\lambda x_{k}^{p_{k}}} x_{k}^{N+[p_{k}]} \tilde{a}_{k}(x_{k},\dots,x_{n}). 
\notag 
\end{align} 
By interchanging integrals, $F_{k2}(\lambda)$ can be expressed as the oscillatory integral 
\begin{align} 
F_{k2}(\lambda) 
&= \int_{0}^{\infty} dx_{k} e^{\pm_{k} i\lambda x_{k}^{p_{k}}} x_{k}^{N+[p_{k}]} b_{k}(x_{k}) 
= \tilde{I}_{p_{k},N+[p_{k}]+1}^{\pm_{k}}[b_{k}](\lambda), 
\notag 
\end{align} 
with 
\begin{align} 
b_{k}(x_{k}) :&= \prod_{j=k+1}^{n} \int_{0}^{\infty} dx_{j} e^{\pm_{j}i\lambda x_{j}^{p_{j}}} \tilde{a}_{k}(x_{k},\dots,x_{n}). 
\notag 
\end{align} 
Thanks to $\tilde{a}_{k} \in C^{\infty}_{0}(\mathbb{R}^{n-k-1})$, since $b_{k} \in C^{\infty}_{0}(\mathbb{R}) \subset \mathcal{A}^{\tau}_{\delta}(\mathbb{R})$, 
by Theorem \ref{Lax02} (v) in \S \ref{section_Existence of oscillatory integrals}, 
there exists a positive constant $M^{(k)}_{2}$ such that for any $\lambda \geq 1$, 
\begin{align} 
|F_{k2}(\lambda)| 
\leq M^{(k)}_{2} \lambda^{-\frac{N+1-(p_{k}-[p_{k}])}{p_{k}}}. 
\label{J_k_2_lambda_est} 
\end{align} 
On the other hands, since $a_{1,\dots,k} \in C^{\infty}_{0}(\mathbb{R}^{n-k})$, $F_{k1}(\lambda)$ satisfies 
\begin{align} 
|F_{k1}(\lambda)| 
\leq M^{(k)}_{1} \lambda^{-\frac{N+1}{p_{k}}} 
\label{J_k_1_lambda_est} 
\end{align} 
for $\lambda \geq 1$, where $M^{(k)}_{1}$ is a positive constant. 

Combining \eqref{J_k_lambda_decomposition} with \eqref{J_k_2_lambda_est} and \eqref{J_k_1_lambda_est} yields 
\begin{align} 
|F_{k}(\lambda)| 
\leq M^{(k)}_{1} \lambda^{-\frac{N+1}{p_{k}}} + M^{(k)}_{2} \lambda^{-\frac{N+1-(p_{k}-[p_{k}])}{p_{k}}} 
\leq M^{(k)} \lambda^{-\frac{N+1-(p_{k}-[p_{k}])}{p_{k}}} 
\notag 
\end{align} 
for $\lambda \geq 1$, where $M^{(k)}:=\max_{h=1,2} M^{(k)}_{h}$. 
Moreover adding \eqref{I_k_1_def} leads to 
\begin{align} 
|I^{(k)}_{1}| 
\leq C_{k} \lambda^{-\sum_{j=1}^{k-1} \frac{1}{p_{j}}} M^{(k)} \lambda^{-\frac{N+1-(p_{k}-[p_{k}])}{p_{k}}} 
\leq C_{k} M^{(k)} \lambda^{-\frac{N+1-(p_{k}-[p_{k}])}{p_{k}}} 
\notag 
\end{align} 
for $\lambda \geq 1$, where $C_{k}$ is a positive constant. 
Thus we have 
\begin{align} 
\Bigg| \sum_{k=1}^{n} I^{(k)}_{1} \Bigg| 
\leq \sum_{k=1}^{n} C_{k} M^{(k)} \lambda^{-\frac{N+1-(p_{k}-[p_{k}])}{p_{k}}} 
\leq M_{n} \lambda^{-\frac{N+1-\max (p_{j}-[p_{j}])}{\max p_{j}}} 
\label{sum_k=1_n_I_k_1_est} 
\end{align} 
for $\lambda \geq 1$, where $M_{n}:=n\max_{k=1,\dots,n} C_{k} M^{(k)}$. 

By \eqref{sum_k=1_n_I_k_1_est}, any term in \eqref{I_n_0_def} where the order of $\lambda$ is less than or equal to $-\frac{N+1-\max (p_{j}-[p_{j}])}{\max p_{j}}$ can be turned into the remainder term. 
In fact, defining $\Omega_{p}^{N}$ by \eqref{Omega_k}, if $\alpha = (\alpha_{1},\dots,\alpha_{n}) \in \Omega_{p}^{N}$, since 
\begin{align} 
\frac{\alpha_{j}+1}{p_{j}} \leq \sum_{j=1}^{n} \frac{\alpha_{j}+1}{p_{j}} < \frac{N+1-\max (p_{j}-[p_{j}])}{\max p_{j}} \leq \frac{N+1}{p_{j}}, 
\notag 
\end{align} 
then 
\begin{align} 
\alpha_{j} < N~\text{for}~j=1.\dots,n,~\text{and}~\lambda^{-\sum_{j=1}^{n} \frac{\alpha_{j}+1}{p_{j}}} > \lambda^{-\frac{N+1-\max (p_{j}-[p_{j}])}{\max p_{j}}}~\text{for}~\lambda \geq 1. 
\notag 
\end{align} 
Therefore for any $N \in \mathbb{N}$, we obtain \eqref{tilde_I_phi_a_lambda} as $\lambda \to \infty$. 

(ii) 
If we replace the use of Theorem \ref{th02} (i) in the proof of (i) with Theorem \ref{th02} (ii), we can prove (ii). 
\end{proof}

\end{document}